\documentclass[12pt,lenq]{amsart}
\usepackage{graphicx}
\usepackage{amsmath}
\usepackage{amscd}
\usepackage{amssymb}
\usepackage{amsbsy}
\usepackage{amsfonts}
\usepackage{latexsym}
\usepackage{graphics}
\usepackage{amsmath,amscd,latexsym}
\usepackage{multirow}
\usepackage{array}
\usepackage{paralist}
\usepackage{titletoc}
\theoremstyle{plain}
\newtheorem{theorem}{Theorem}[section]
\newtheorem{proposition}[theorem]{Proposition}
\newtheorem{lemma}[theorem]{Lemma}
\newtheorem{corollary}[theorem]{Corollary}
\newtheorem{remark}[theorem]{Remark}
\newtheorem{definition}[theorem]{Definition}
\newtheorem{notation}[theorem]{Notation}

\newtheorem{main theorem}[theorem]{Main Theorem}

\newtheorem{question}[theorem]{Question}
\newtheorem{convention}[theorem]{Convention}

\newcommand{\ZZ}{\mathbb{Z}}

\newcommand{\lp}{(\hskip -0.07cm (}
\newcommand{\rp}{)\hskip -0.07cm )}

\begin{document}

\title{Homotopically equivalent simple loops
on 2-bridge spheres in 2-bridge link complements (III)}
\author{Donghi Lee}
\address{Department of Mathematics\\
Pusan National University \\
San-30 Jangjeon-Dong, Geumjung-Gu, Pusan, 609-735, Republic of Korea}
\email{donghi@pusan.ac.kr}

\author{Makoto Sakuma}
\address{Department of Mathematics\\
Graduate School of Science\\
Hiroshima University\\
Higashi-Hiroshima, 739-8526, Japan}
\email{sakuma@math.sci.hiroshima-u.ac.jp}

\subjclass[2010]{Primary 20F06, 57M25 \\
\indent {The first author was supported by Basic Science Research Program
through the National Research Foundation of Korea(NRF) funded
by the Ministry of Education, Science and Technology(2012R1A1A3009996).
The second author was supported
by JSPS Grants-in-Aid 22340013 and 21654011.}}

\begin{abstract}
This is the last of a series of papers which give a necessary and sufficient condition for two essential simple loops on a 2-bridge sphere in a 2-bridge link complement to be homotopic in the link complement. The first paper of the series treated the case of the 2-bridge torus links, and the second paper treated the case of 2-bridge links of slope $n/(2n+1)$ and $(n+1)/(3n+2)$, where $n \ge 2$ is an arbitrary integer. In this paper, we first treat the case of 2-bridge links of slope $n/(mn+1)$ and $(n+1)/((m+1)n+m)$, where $m \ge 3$ is an arbitrary integer, and then treat the remaining cases by induction.
\end{abstract}
\maketitle

\section{Introduction}

Let $K$ be a 2-bridge link in $S^3$ and
let $S$ be a 4-punctured sphere in $S^3-K$
obtained from a $2$-bridge sphere of $K$.
The present paper is a continuation
of \cite{lee_sakuma_2} and \cite{lee_sakuma_3}
in a series of papers
which give a necessary and sufficient condition
for two essential simple loops on $S$
to be homotopic in $S^3-K$.
Ahead of this series,
the authors~\cite{lee_sakuma} gave a complete characterization of
those essential simple loops in $S$
which are null-homotopic in $S^3-K$.

The first paper~\cite{lee_sakuma_2}
of the series treated the case of a $2$-bridge link of slope $1/p=[p]$,
and the second paper~\cite{lee_sakuma_3} treated the case
of a 2-bridge link of slope $n/(2n+1)=[2,n]$
or slope $(n+1)/(3n+2)=[2,1,n]$.
On the other hand,
the first half of the present paper treats the case of a 2-bridge link
of slope $n/(mn+1)=[m,n]$ or slope $(n+1)/((m+1)n+m)=[m,1,n]$,
where $m \ge 3$ is an arbitrary integer.
These five families play special roles in our project in the sense that
the treatment of these links form a base step
of an inductive proof of the main theorem
for a $2$-bridge link of general slope $[m_1,m_2, \dots,m_k]$
to which the second half of the present paper contributes,
where the induction uses $k \ge 1$ as the parameter.

In the present paper,
we also give a complete characterization of those simple loops
in the $2$-bridge sphere of a hyperbolic $2$-bridge link
to be peripheral or imprimitive in the link complement
(see Theorems~\ref{main_corollary} and \ref{main_corollary2}).

This paper is organized as follows.
In Section~\ref{sec:main_result},
we describe the main results of this paper (Main Theorem~\ref{main_theorem}
and Theorems~\ref{main_corollary} and~\ref{main_corollary2}).
In Section~\ref{sec:technical_lemmas}, we establish
technical lemmas used for the proofs in Sections~\ref{sec:proof_of_main_theorem_1}
and \ref{sec:proof_of_main_theorem_3}.
Two special cases of Main Theorem~\ref{main_theorem}
namely, the cases of a 2-bridge link of slope $n/(mn+1)=[m,n]$
and a $2$-bridge link of slope $(n+1)/((m+1)n+m)=[m,1,n]$, where $m,n \ge 3$,
are treated in Sections~\ref{sec:proof_of_main_theorem_1} and
\ref{sec:proof_of_main_theorem_3}, respectively.
For the case of a $2$-bridge link of general slope,
we start with preliminary results in Section~\ref{sec:technical_lemmas_general}
and perform transformation so that we may apply the induction as discussed
in Section~\ref{sec:transformation}.
In Section~\ref{sec:result_for_induction}, we prove key results
for the induction, and finally the proof of Main Theorem~\ref{main_theorem}
for the general cases is contained in Section~\ref{sec:proof_for_general_2-bridge_links}.
In Section~\ref{sec:proof_of_main_corollary},
we prove Theorems~\ref{main_corollary} and \ref{main_corollary2}.

\section{Main result}
\label{sec:main_result}

This paper, as a continuation of \cite{lee_sakuma_2} and \cite{lee_sakuma_3},
uses the same notation and terminology as in
\cite{lee_sakuma_2} and \cite{lee_sakuma_3} without specifically mentioning.
We begin with the following question,
providing whose answer is the purpose of this series of papers.

\begin{question} \label{question}
Consider a $2$-bridge link $K(r)$ with $r\ne \infty$.
For two distinct rational numbers $s, s' \in I_1(r) \cup I_2(r)$,
when are the unoriented loops $\alpha_s$ and $\alpha_{s'}$
homotopic in $S^3-K(r)$?
\end{question}

By Schubert's classification of $2$-bridge links~\cite{Schubert},
we may assume that
$r$ is a rational number with $0 \le r \le 1/2$ or $r=\infty$.
If $r=0$ or $\infty$,
then $G(K(r))$ is the infinite cyclic group or the rank $2$ free group accordingly,
and we can easily obtain an answer to Question~\ref{question}
(see~\cite[Paragraph after Question~2.2]{lee_sakuma_2}).
So we may assume $0< r \le 1/2$.
In the first paper~\cite{lee_sakuma_2} and the second paper~\cite{lee_sakuma_3}
of this series, we gave a complete answer to the question,
respectively, for $r=1/p$ with $p \ge 2$
and for $r=n/(2n+1)=[2,n]$ or $r=(n+1)/(3n+2)=[2,1,n]$ with $n \ge 2$
(see \cite[Main Theorem~2.7]{lee_sakuma_2} and
\cite[Main Theorems~2.2 and 2.3]{lee_sakuma_3}).
Since there exist a homeomorphism from $(S^3,K(n/(2n+1)))$ to $(S^3, K(2/(2n+1)))$
and a homeomorphism from $(S^3,K((n+1)/(3n+2)))$ to $(S^3, K(3/(3n+2)))$
both of which send the upper/lower tangles to lower/upper tangles,
we obtain, from \cite[Main Theorems~2.2 and 2.3]{lee_sakuma_3},
an answer to Question~\ref{question} for $K(r)$
with $r=2/(2n+1)=[n, 2]$ or $r=3/(3n+2)=[n, 1, 2]$.

In the present paper, we solve
Question~\ref{question} for the remaining cases.

\begin{main theorem}
\label{main_theorem}
Suppose that $r$ is a rational number with $0 < r \le 1/2$ such that
$r \neq 1/n$, $r \neq n/(2n+1)$, $r \neq 2/(2n+1)$,
$r \neq (n+1)/(3n+2)$ and $r \neq 3/(3n+2)$,
where $n \ge 2$ is an integer.
Then, for any two distinct rational numbers $s, s' \in I_1(r) \cup I_2(r)$,
the unoriented loops $\alpha_s$ and $\alpha_{s'}$ are never
homotopic in $S^3-K(r)$.
\end{main theorem}

This theorem together with \cite[Main Theorem~2.7]{lee_sakuma_2} and
\cite[Main Theorems~2.2 and 2.3]{lee_sakuma_3}
implies the following complete answer to Question~\ref{question}.

\begin{theorem}
\label{summary_theorem}
Suppose that $r$ is a rational number such that $0 < r \le 1/2$.
For distinct $s, s' \in I_1(r)\cup I_2(r)$,
the unoriented loops $\alpha_s$ and $\alpha_{s'}$ are
homotopic in $S^3-K(r)$
if and only if one of the following holds.
\begin{enumerate}[\indent \rm (1)]
\item
$r=1/p$, where $p \ge 2$ is an integer,
and $s=q_1/p_1$ and $s'=q_2/p_2$ satisfy
$q_1=q_2$ and $q_1/(p_1+p_2)=1/p$, where $(p_i, q_i)$ is a pair of
relatively prime positive integers.
\item
$r=3/8$, namely $K(r)$ is the Whitehead link,
and the set $\{s, s'\}$ equals
either $\{1/6, 3/10\}$ or $\{3/4, 5/12\}$.
\end{enumerate}
\end{theorem}

The proof of the main theorem together with
\cite[Theorems~2.5 and 2.6]{lee_sakuma_3}
implies the following theorems,
which give a complete characterization of those simple loops
in the $2$-bridge sphere of a hyperbolic $2$-bridge link
to be peripheral or imprimitive in the link complement.

\begin{theorem}
\label{main_corollary}
Suppose that $r$ is a rational number with $0 < r \le 1/2$
such that $r \neq 1/n$ for any integer $n \ge 2$.
For a rational number $s \in I_1(r) \cup I_2(r)$,
the loop $\alpha_s$ is peripheral
if and only if one of the following holds.
\begin{enumerate}[\indent \rm (1)]
\item
$r=2/5$ and $s=1/5$ or
$s=3/5$.
\item
$r=n/(2n+1)=[2, n]$ for some integer $n \ge 3$,
and $s=(n+1)/(2n+1)$.
\item
$r=2/(2n+1)=[n,2]$ for some integers $n \ge 3$,
and $s=1/(2n+1)$.
\end{enumerate}
\end{theorem}

\begin{theorem}
\label{main_corollary2}
Suppose that $r$ is a rational number with $0 < r \le 1/2$
such that $r \neq 1/n$, where $n \ge 2$ is an integer.
Then, for a rational number $s \in I_1(r) \cup I_2(r)$,
the free homotopy class $\alpha_s$ is primitive
with the following exceptions.
\begin{enumerate}[\indent \rm (1)]
\item $r=2/5$, and $s=2/7$ or $s=3/4$.
In this case, $\alpha_s$ is the third power of some primitive element
in $G(K(r))$.

\item $r=3/7$ and $s=2/7$.
In this case, $\alpha_s$ is the second power of some primitive element
in $G(K(r))$.

\item $r=2/7$ and $s=3/7$.
In this case, $\alpha_s$ is the second power of some primitive element
in $G(K(r))$.
\end{enumerate}
\end{theorem}

We prove the above main theorem
by interpreting the situation
in terms of combinatorial group theory.
In other words, we prove that
two words representing the free homotopy classes of
$\alpha_s$ and $\alpha_{s'}$
are never conjugate in the $2$-bridge link group $G(K(r))$
for any two distinct rational numbers $s, s' \in I_1(r) \cup I_2(r)$.
The key tool used in the proof is small cancellation theory,
applied to two-generator and one-relator presentations
of $2$-bridge link groups.

\section{Technical lemmas for $r=[m, n]$ or $r=[m, 1, n]$ with
$m, n \ge 3$}
\label{sec:technical_lemmas}

Throughout the remainder of the this paper,
we simply write Hypotheses~A, B and C
instead of \cite[Hypothesis~A]{lee_sakuma_3}, \cite[Hypothesis~B]{lee_sakuma_3} and
\cite[Hypothesis~C]{lee_sakuma_3}, respectively.
Throughout this section, we assume that Hypothesis~A holds.
Then by \cite[Lemma~3.3]{lee_sakuma_3}, either \cite[Lemma~3.3(1)]{lee_sakuma_3}
or \cite[Lemma~3.3(2)]{lee_sakuma_3} holds, that is, either Hypothesis~B
or Hypothesis~C holds.
Accordingly as Hypothesis~B or Hypothesis~C holds,
we shall establish several technical lemmas
used for the proof of Main Theorem~\ref{main_theorem}
for $r=[m, n]$ and $r=[m, 1, n]$
in Sections~\ref{sec:proof_of_main_theorem_1}
and \ref{sec:proof_of_main_theorem_3}, respectively.

\subsection{The case when Hypothesis~B holds}

We first assume that Hypothesis~B holds. We begin with the following remark
before introducing technical lemmas concerning the cyclic sequence
$CS(\phi(\alpha))=CS(u_s)=CS(s)$.

\begin{remark}
\label{rem:(1)holds}
{\rm
(1) If $r=[m, n]$, where $m, n \ge 3$ are integers,
then, by \cite[Lemma~3.16(3)]{lee_sakuma_2},
$CS(r)=\lp m+1, (n-1) \langle m \rangle, m+1, (n-1) \langle m \rangle \rp$,
where $S_1=(m+1)$ and $S_2=((n-1) \langle m \rangle)$.
So, in Hypothesis~B,
both $S(\phi(\partial D_i^+))$ and $S(\phi(\partial D_i^-))$
are exactly of the form
$(\ell_1, n_1 \langle m \rangle, m+1, n_2 \langle m \rangle, \ell_2)$,
where $0 \le \ell_1, \ell_2 \le m-1$ and $0 \le n_1, n_2 \le n-1$
are integers such that if $n_j=n-1$ then $\ell_j$ is necessarily $0$
for $j=1, 2$. In particular, $S(y_{i,b})=(\ell)$ with $1 \le \ell \le m$,
unless $y_i$ is an empty word. The same is true for $S(z_{i,e})$, $S(y_{i,b}')$
and $S(z_{i,e}')$.

(2) If $r=[m, 1, n]$, where $m, n\ge 3$ are integers,
then, by \cite[Lemma~3.16(1)]{lee_sakuma_2},
$CS(r)=\lp n \langle m+1 \rangle, m, n \langle m+1 \rangle, m \rp$,
where $S_1=(n \langle m+1 \rangle)$ and $S_2=(m)$.
So, in Hypothesis~B,
both $S(\phi(\partial D_i^+))$ and $S(\phi(\partial D_i^-))$
are exactly of the form $(\ell_1, n \langle m+1 \rangle, \ell_2)$,
where $0 \le \ell_1, \ell_2 \le m$ are integers.
In particular, $S(w_{i,b})=S(w_{i,e})=(m+1)$ and $S(w_{i,b}')=S(w_{i,e}')=(m+1)$.
}
\end{remark}

\begin{lemma}
\label{lem:case1-1(a)}
Let $r=[m,n]$, where $m, n \ge 3$ are integers.
Under Hypothesis~B,
suppose that $v$ is a subword of the cyclic word represented by
$\phi(\alpha) \equiv y_1 w_1 z_1 y_2 w_2 z_2 \cdots y_t w_t z_t$
such that $v$ corresponds to a term of $CS(\phi(\alpha))=CS(s)$.
Then, after a cyclic shift of indices,
$v$ is equal to one of the following subwords:
\[
z_{0, e}w_1w_2\cdots w_qy_{q+1, b}, \quad
z_{0, e}w_1w_2\cdots w_q, \quad
w_1w_2\cdots w_qy_{q+1, b}, \quad
w_1w_2\cdots w_q,
\]
where $q\in \ZZ_+\cup\{0\}$ in the first three cases
and $q\in \ZZ_+$ in the last case.
In each of the above,
the ``intermediate subwords'' are empty;
to be precise, when we say that
$z_{0, e}w_1w_2\cdots w_qy_{q+1, b}$,
for example,
is a subword of $(u_s)$,
we assume that $y_1$, $z_iy_{i+1}$ ($1\le i\le q-1$) and $z_q$
are empty words.
\end{lemma}

\begin{proof}
The proof can be done in the same way as in \cite[Lemma~3.6]{lee_sakuma_3}.
\end{proof}

Throughout the remainder of this paper, we will assume the following convention.

\begin{convention}
\label{con:figure}
{\rm
In Figures~\ref{fig.lemma-1-1}--\ref{fig.general_proof_1},
except for Figure~\ref{fig.converging},
the change of directions of consecutive arrowheads
represents the change from positive (negative, resp.) words
to negative (positive, resp.) words, and
a dot represents a vertex whose position is clearly identified.
Also an Arabic number represents the length of the corresponding positive
(or negative) word.
In Figures~\ref{fig.Lemma7_3_1}--\ref{fig.Lemma7_3_2}
and \ref{fig.S_1_occurs}--\ref{fig.S_1_occurs(c)},
the label $S_1$ or $S_2$ on an oriented segment
means that the $S$-sequence of the corresponding word
is equal to $S_1$ or $S_2$ accordingly.}
\end{convention}

\begin{lemma}
\label{lem:case1-1(b)}
Let $r=[m,n]$, where $m, n \ge 3$ are integers.
Under Hypothesis~B, the following hold for every $i$.
\begin{enumerate}[\indent \rm (1)]
\item
$S(z_{i, e}y_{i+1, b}) \ne (m+d)$ for any integer $d$ with $1 \le d \le m-1$.

\item
$S(w_iz_iy_{i+1, b}) \ne (m+1+d)$ and
$S(z_{i, e}y_{i+1}w_{i+1}) \ne (m+1+d)$
for any integer $d$ with $1 \le d \le m-1$.

\item
$S(w_iz_iy_{i+1, b}) \neq (2m+1)$
and $S(z_{i, e}y_{i+1}w_{i+1}) \neq (2m+1)$.

\item
$S(w_iz_iy_{i+1}w_{i+1}) \ne (2m+2)$.

\item $S(w_iz_iy_{i+1}w_{i+1}) \ne (m+1, m+1)$.
\end{enumerate}
\end{lemma}

\begin{proof}
(1) Suppose on the contrary that the assertion does not hold.
Without loss of generality, we may assume that
$S(z_{1,e}y_{2,b})=(m+d)$ for some $1 \le d \le m-1$.
Then, since $0 \le |z_{1,e}|, |y_{2,b}| \le m$,
we have $|z_{1,e}|=k$ and $|y_{2,b}|=m+d-k$ for some
$k$ with $d \le k \le m$.
Here, we assume $d\le k\le m-1$ and hence $|z_{1,e}'|=m-k \ne 0$.
(The case when $k=m$ and hence $|z_{1,e}'|= 0$ can be treated similarly.)
Then $J$ is locally as illustrated in Figure~\ref{fig.lemma-1-1}(a)
which follows Convention~\ref{con:figure}.
The numbers $k$ and $m+d-k$ near the upper boundary
represent the lengths of the words $z_{1,e}$ and $y_{2,b}$,
respectively, whereas the numbers $m-k$ and $k-d$ near the lower boundary
represent the lengths of the words $z_{1,e}'$ and $y_{2,b}'$, respectively,
and the change of directions of consecutive arrowheads
represents the change from positive (negative, resp.) words
to negative (positive, resp.) words.

Suppose first that $J=M$.
Then we see from Figure~\ref{fig.lemma-1-1}(a) that
$CS(\phi(\delta^{-1}))=CS(u_{s'}^{\pm 1})=CS(s')$
involves a term $(m-k)+(k-d)=m-d$.
Moreover, $CS(s')$ also
involves a term of the form $m+1+c$
with $c \in \ZZ_+ \cup \{0\}$,
because $S(w_1')=S_1=(m+1)$.
Since $m-d \le m-1$, this is a
contradiction to \cite[Lemma~3.8]{lee_sakuma_2}
which says that either $CS(s')$ is equal to $\lp \ell,\ell \rp$
or $CS(s')$ consists of $\ell$ and $\ell+1$ for some $\ell \in \ZZ_+$.
Suppose next that $J \subsetneq M$.
Since none of $S(\phi(e_{1}'))$ and $S(\phi(e_{2}'))$ contains
$S_1=(m+1)$ as a subsequence (see \cite[Lemma~3.1(1)]{lee_sakuma_3}),
we see that the initial vertex of $e_2'$ lies in the interior
of the segment of $\partial D_1^-$
corresponding to $S_1=(m+1)$.
Similarly, the terminal vertex of $e_3'$
lies in the interior of the segment of $\partial D_2^-$
corresponding to $S_1=(m+1)$.
Hence, we see from Figure~\ref{fig.lemma-1-1}(b)
that $S(\phi(e_2'e_3'))$ contains a subsequence of the form $(\ell_1, m-d, \ell_2)$
with $\ell_1, \ell_2 \in \ZZ_+$.
This yields that a term $m-d$ occurs in
$CS(\phi(\partial D_1'))=CS(r)=\lp m+1, (n-1)\langle m \rangle, m+1, (n-1)\langle m \rangle\rp$,
which is obviously a contradiction.

\begin{figure}[h]
\includegraphics{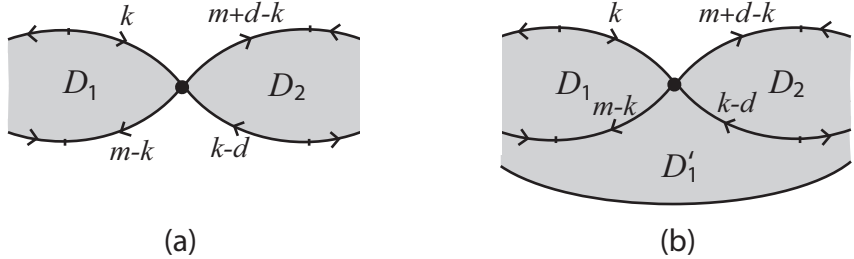}
\caption{
Lemma~\ref{lem:case1-1(b)}(1) where $S(z_{1,e}y_{2,b})=(k+(m+d-k))$}
\label{fig.lemma-1-1}
\end{figure}

(2) Suppose on the contrary that $S(w_1z_1y_{2, b})=(m+1+d)$
for some $1 \le d \le m-1$.
(The other case is treated similarly.)
Then, since $w_1$ and $z_1$ have different signs
when $z_1$ is nonempty
and since $|w_1|=m+1$ and $0 \le |y_{2,b}| \le m$,
the only possibility is that $|z_1|=0$, $|y_{2, b}|=d$
and $S(w_1y_{2,b})=(m+1+d)$.
If $J=M$, then we see from Figure~\ref{fig.lemma-1-2}(a) that
$CS(s')$ involves both a term $m-d$ and a term of the form $m+1+c$
with $c \in \ZZ_+ \cup \{0\}$. Since $m-d \le m-1$, this is a
contradiction to \cite[Lemma~3.8]{lee_sakuma_2}.
On the other hand, if $J \subsetneq M$,
then we see, by using \cite[Lemma~3.1(1)]{lee_sakuma_3}
as in the proof of Lemma~\ref{lem:case1-1(b)}(1),
that $S(\phi(e_2'e_3'))$ contains a subsequence of the form $(\ell_1, m-d, \ell_2)$
with $\ell_1, \ell_2 \in \ZZ_+$ (see Figure~\ref{fig.lemma-1-2}(b)).
This implies that $CS(\phi(\partial D_1'))=CS(r)$ has a term $m-d$,
a contradiction.

\begin{figure}[h]
\includegraphics{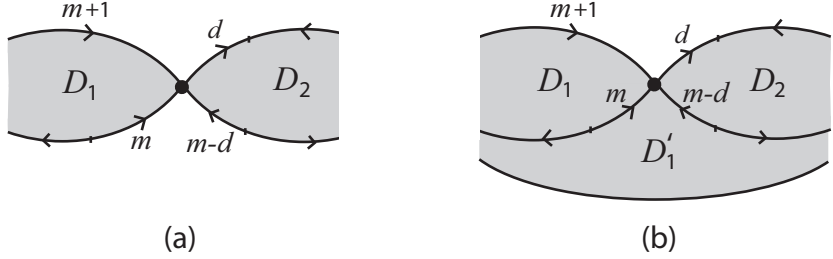}
\caption{
Lemma~\ref{lem:case1-1(b)}(2) where $S(w_1z_1y_{2, b})=((m+1)+0+d)$}
\label{fig.lemma-1-2}
\end{figure}

(3) Suppose on the contrary that $S(z_{1,e}y_2w_2)=(2m+1)$.
(The other case is treated similarly.)
Then $|z_{1,e}|=m$ and $|y_2|=0$.
We note that $z_1=z_{1,e}$.
Otherwise, $S(w_1z_1)$ contains a subword $(m+1,m,m)$
and hence $CS(\phi(\alpha))=CS(s)$ contains a term $m$ and $2m+1$,
a contradiction to \cite[Lemma~3.8]{lee_sakuma_2}.
Hence, we see that $CS(w_1'z_1'y_2'w_2')$ contains a term $2m$
(see Figure~\ref{fig.lemma-1-3}(a)).
If $J=M$,
then we see from Figure~\ref{fig.lemma-1-3}(a) that
$CS(\phi(\delta^{-1}))=CS(s')$ contains both
a term $m$ and a term $2m$. Since $m \ge 3$ implying that $m+2 \le 2m$,
this gives a contradiction to \cite[Lemma~3.8]{lee_sakuma_2}.
On the other hand, if $J \subsetneq M$,
then we see, by using \cite[Lemma~3.1(1)]{lee_sakuma_3}
as in the proof of Lemma~\ref{lem:case1-1(b)}(1),
that $S(\phi(e_2'e_3'))$ is of the form $(\ell_1, 2m, \ell_2)$
with $\ell_1, \ell_2 \in \ZZ_+$.
This implies that a term $2m$ occurs
in $CS(\phi(\partial D_1'))=CS(r)$, which is a contradiction.

(4) This can be proved by an argument parallel to the proof of (3)
(see Figure~\ref{fig.lemma-1-3}(b)).

\begin{figure}[h]
\includegraphics{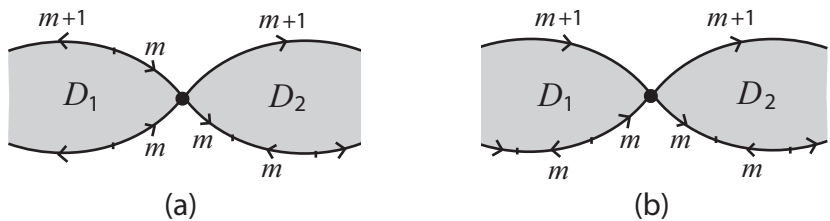}
\caption{
(a) Lemma~\ref{lem:case1-1(b)}(3) where $S(z_{1,e}y_2w_2)=(m+0+(m+1))$,
and (b) Lemma~\ref{lem:case1-1(b)}(4) where $S(w_1z_1y_2w_2)=((m+1)+0+0+(m+1))$}
\label{fig.lemma-1-3}
\end{figure}

(5) Suppose on the contrary that $S(w_1z_1y_2w_2)=(m+1, m+1)$.
Then $|z_1|=|y_2|=0$ and $S(w_1w_2)=(m+1, m+1)$.
If $J=M$ (see Figure~\ref{fig.lemma-1-8}(a)),
then $CS(s')$ contains both a subsequence
$((n-1) \langle m \rangle)$ and a term of the form $m+1+c$
with $c \in \ZZ_+ \cup \{0\}$.
Here, if $c=0$, then $s' \notin I_1(r) \cup I_2(r)$ by \cite[Proposition~3.19(1)]{lee_sakuma_2},
contradicting the hypothesis of the theorem,
while if $c>0$, then we have a contradiction to \cite[Lemma~3.8]{lee_sakuma_2}.
On the other hand, if $J \subsetneq M$ (see Figure~\ref{fig.lemma-1-8}(b)),
then we see, by using \cite[Lemma~3.1(1)]{lee_sakuma_3}
as in the proof of Lemma~\ref{lem:case1-1(b)}(1),
that $S(\phi(e_2'e_3'))$ is of the form $(\ell_1, 2(n-1) \langle m \rangle, \ell_2)$
with $\ell_1, \ell_2 \in \ZZ_+$.
This implies that a subsequence $(2(n-1) \langle m \rangle)$ occurs
in $CS(\phi(\partial D_1'))=CS(r)$, which is a contradiction.
\end{proof}

\begin{figure}[h]
\includegraphics{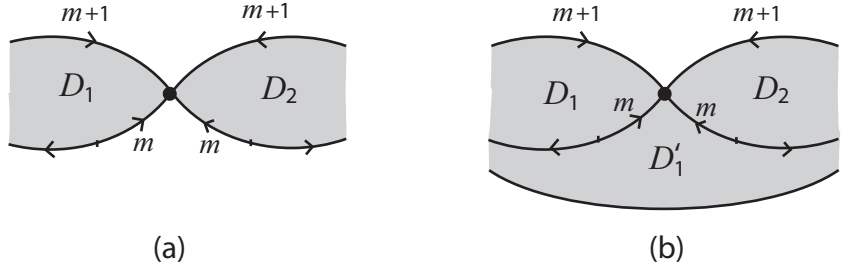}
\caption{
Lemma~\ref{lem:case1-1(b)}(5)
where $S(w_1z_1y_2w_2)=(m+1, m+1)$}
\label{fig.lemma-1-8}
\end{figure}

\begin{lemma}
\label{lem:case1-1}
Let $r=[m,n]$, where $m, n \ge 3$ are integers.
Under Hypothesis~B,
the following hold.
\begin{enumerate}[\indent \rm (1)]
\item No two consecutive terms of $CS(s)$ can be $(m+1, m+1)$.

\item No term of $CS(s)$ can be of the form $m+1+d$ except $2m$,
where $d \in \ZZ_+$.

\item No two consecutive terms of $CS(s)$ can be $(2m, 2m)$.
\end{enumerate}
\end{lemma}

\begin{proof}
(1) Suppose on the contrary that $CS(\phi(\alpha))=CS(u_s)=CS(s)$
contains $(m+1, m+1)$ as a subsequence.
Let $v=v'v''$ be a subword of the cyclic word $(u_s)$
corresponding to a subsequence $(m+1, m+1)$,
where $S(v')=S(v'')=(m+1)$.
By using Lemma~\ref{lem:case1-1(a)}
and the facts that $0 \le |z_{i,e}|, |y_{i,b}| \le m$ and $|w_i|=m+1$,
we see that one
of the following holds after a shift of indices.
\begin{enumerate}[\indent \rm (i)]
\item $(v',v'')=(z_{1,e}y_{2,b},w_2)$, where $S(z_{1,e}y_{2,b})=(m+1)$.
\item $(v',v'')=(w_1,z_{1,e}y_{2,b})$, where $S(z_{1,e}y_{2,b})=(m+1)$.
\item $(v',v'')=(w_1,w_2)$.
\end{enumerate}
However, (i) and (ii) are impossible
by Lemma~\ref{lem:case1-1(b)}(1),
and (iii) is impossible
by Lemma~\ref{lem:case1-1(b)}(5).

(2) Suppose on the contrary that $CS(s)$ contains a term $m+1+d$
except $2m$.
Let $v$ be a subword of the cyclic word $(u_s)$
corresponding to a term $m+1+d$ except $2m$.
By using Lemma~\ref{lem:case1-1(a)}
and the facts that $0 \le |z_{i,e}|, |y_{i,b}| \le m$ and $|w_i|=m+1$,
we see that one of the following holds
after a cyclic shift of indices.
\begin{enumerate}[\indent \rm (i)]
\item $v=z_{0, e}y_{1, b}$ with $|z_{0, e}|, |y_{1, b}| \neq 0$.

\item $v$ contains $z_{0, e}w_1$ with $|z_{0, e}| \neq 0$.

\item $v$ contains $w_1y_{2, b}$ with $|y_{2, b}| \neq 0$.

\item $v$ contains $w_1w_2$ with $|z_1|=|y_2|=0$.
\end{enumerate}
However, (i) is impossible by
Lemma~\ref{lem:case1-1(b)}(1).
By Lemma~\ref{lem:case1-1(b)}(2), $|z_{0, e}|=m$ provided (ii) occurs,
and $|y_{2, b}|=m$ provided (iii) occurs. Hence (ii) and (iii) are impossible
by Lemma~\ref{lem:case1-1(b)}(3). Also (iv) is impossible
by Lemma~\ref{lem:case1-1(b)}(4).

(3) Suppose on the contrary that $CS(s)$
contains $(2m, 2m)$ as a subsequence.
Let $v=v'v''$ be a subword of the cyclic word $(u_s)$
corresponding to a subsequence $(2m, 2m)$,
where $S(v')=S(v'')=(2m)$.
If $v'$ contains $w_i$ for some $i$,
then we see, by using Lemma~\ref{lem:case1-1(a)}
and the identity $|w_i|=m+1$, that
either $v'=w_iz_iy_{i+1,b}$ with
$(|z_i|,|y_{i+1,b}|)=(0,m-1)$
or $v'=z_{i-1,e}y_iw_i$ with
$(|z_{i-1}|,|y_i|)=(m-1,0)$.
However both cases are impossible by
Lemma~\ref{lem:case1-1(b)}(2).
Thus $v'$ cannot contain $w_i$.
Since $S(v')=(2m)$ is a term of $CS(s)$,
this implies that $v'$ is disjoint from $w_i$
for every $i$.
The same conclusion also holds for $v''$,
and hence for $v=v'v''$.
Thus $v$ is a subword of $z_iy_{i+1}$ for some $i$.
But then $S(v)$ contains a term $2m$ at most once,
a contradiction.
\end{proof}

\begin{corollary}
\label{cor:case1-1}
Let $r=[m,n]$, where $m, n \ge 3$ are integers.
Under Hypothesis~B,
$CS(s)$ consists of $m$ and $m+1$.
\end{corollary}

\begin{proof}
By \cite[Lemma~3.8]{lee_sakuma_2},
either $CS(s)=\lp \ell, \ell \rp$ or $CS(s)$ consists of $\ell$ and $\ell+1$
with $\ell \in \ZZ_+$.
By Hypothesis~B together with
Remark~\ref{rem:(1)holds}(1), $\phi(\alpha)$ involves
a subword $w_i$ whose $S$-sequence is $(m+1)$,
so $CS(\phi(\alpha))=CS(u_s)=CS(s)$ must contain
a term of the form $m+1+c$, where $c \in \ZZ_+ \cup \{0\}$.
First consider the case where $CS(s)=\lp \ell, \ell \rp$.
Since $CS(s)$ contains a term $m+1+c$, we have $\ell \ge m+1$.
Here, by Lemma~\ref{lem:case1-1}(1), $\ell$ is not equal to $m+1$.
By Lemma~\ref{lem:case1-1}(2),
$\ell$ is not equal to $m+1+d$ for any $d \in \ZZ_+$
except $2m$. However, Lemma~\ref{lem:case1-1}(3)
implies that $\ell$ is not equal to $2m$, so that
there remains no possibility for $\ell$.
Next consider the case where $CS(s)$ consists of $\ell$ and $\ell+1$.
By Lemma~\ref{lem:case1-1}(2),
none of $\ell$ and $\ell+1$ is equal to
$m+1+d$ for any $d \in \ZZ_+$ except $2m$. So $\ell \le m$.
On the other hand, since $CS(s)$ contains a term $m+1+c$,
we have $\ell+1 \ge m+1$. Therefore $\ell=m$.
\end{proof}

Next, we study the case where $r=[m, 1, n]$
with $m, n \ge 3$.
Recall from Remark~\ref{rem:(1)holds}(2) that
$CS(r)=\lp n \langle m+1 \rangle, m, n \langle m+1 \rangle, m \rp$,
where $S_1=(n \langle m+1 \rangle)$ and $S_2=(m)$.
Recall also $S(w_{i,b})=S(w_{i,e})=(m+1)$
and $S(w_{i,b}')=S(w_{i,e}')=(m+1)$ for every $i$.

\begin{lemma}
\label{lem:case3-1_easy}
Let $r=[m,1,n]$, where $m, n \ge 3$ are integers.
Under Hypothesis~B, $CS(s)$ contains $m+1$ as a term.
\end{lemma}

\begin{proof}
The assertion immediately follows from the fact that
$\phi(\alpha)=CS(u_s)=CS(s)$ involves a subword $w_i$ whose $S$-sequence is
$(n \langle m+1 \rangle)$ with $n \ge 3$.
\end{proof}

The following lemma is a counterpart of Lemma~\ref{lem:case1-1(b)}.

\begin{lemma}
\label{lem:case3-1(b)}
Let $r=[m,1,n]$, where $m, n \ge 3$ are integers.
Under Hypothesis~B, the following hold for every $i$.
\begin{enumerate}[\indent \rm (1)]
\item $S(z_iy_{i+1}) \neq (m+2)$.

\item $S(w_{i,e}z_iy_{i+1}) \neq (m+2)$
and $S(z_iy_{i+1}w_{i+1,b}) \neq (m+2)$.
\end{enumerate}
\end{lemma}

\begin{proof}
The proofs of (1) and (2) are parallel to those of
Lemma~\ref{lem:case1-1(b)}(1) and (2), respectively.
\end{proof}

\begin{lemma}
\label{lem:case3-1}
Let $r=[m,1,n]$, where $m, n \ge 3$ are integers.
Under Hypothesis~B, no term of $CS(s)$ can be of the form $m+2$.
\end{lemma}

\begin{proof}
Suppose on the contrary that $CS(s)$ contains a term of the form $m+2$.
Let $v$ be a subword of the cyclic word $(u_s)$
corresponding to a term $m+2$.
Without loss of generality, we may assume that
\begin{enumerate}[\indent \rm (i)]
\item $v=z_0y_1$ with $|z_0|, |y_1| \neq 0$;

\item $v=z_0w_{1, b}$ with $|z_0| \neq 0$; or

\item $v=w_{1, e}y_2$ with $|y_2| \neq 0$.
\end{enumerate}
However, (i) is impossible by Lemma~\ref{lem:case3-1(b)}(1),
and (ii) and (iii) are impossible by Lemma~\ref{lem:case3-1(b)}(2).
\end{proof}

\begin{corollary}
\label{cor:case2-1}
Let $r=[m,1,n]$, where $m, n \ge 3$ are integers.
Under Hypothesis~B,
either $CS(s)=\lp m+1, m+1 \rp$ or $CS(s)$ consists of $m$ and $m+1$.
\end{corollary}

\begin{proof}
By Lemma~\ref{lem:case3-1_easy}, $CS(s)$ contains a term $m+1$.
Also by Lemma~\ref{lem:case3-1}, $CS(s)$ does not contain a term $m+2$.
Hence, by \cite[Lemma~3.8]{lee_sakuma_2},
we obtain the desired result.
\end{proof}

\subsection{The case when Hypothesis~C holds}

We next assume that Hypothesis~C holds. We also begin with the following
remark before introducing two technical lemmas concerning the sequence
$S(z_iy_{i+1})$ accordingly as $r=[m, n]$ and $r=[m, 1, n]$.

\begin{remark}
\label{rem:(2)holds}
{\rm
(1) If $r=[m, n]$, where $m, n \ge 3$ are integers,
then $CS(r)=\lp m+1, (n-1) \langle m \rangle, m+1, (n-1) \langle m \rangle \rp$,
where $S_1=(m+1)$ and $S_2=((n-1) \langle m \rangle)$.
So, in Hypothesis~C,
both $S(\phi(\partial D_i^+))$ and $S(\phi(\partial D_i^-))$
are exactly of the form $(\ell_1, (n-1) \langle m \rangle, \ell_2)$,
where $1 \le \ell_1, \ell_2 \le m$ are integers.

(2) If $r=[m, 1, n]$, where $m, n \ge 3$ are integers,
then $CS(r)=\lp n \langle m+1 \rangle, m, n \langle m+1 \rangle, m \rp$,
where $S_1=(n \langle m+1 \rangle)$ and $S_2=(m)$.
So, in Hypothesis~C,
both $S(\phi(\partial D_i^+))$ and $S(\phi(\partial D_i^-))$
are exactly of the form
$(\ell_1, n_1 \langle m+1 \rangle, m, n_2 \langle m+1 \rangle, \ell_2)$,
where $0 \le \ell_1, \ell_2 \le m$ and $0 \le n_1, n_2 \le n-1$ are integers
such that a pair $(\ell_j, n_j)$ cannot be $(0, 0)$ for $j=1,2$.
In particular, $S(y_{i,b})=(\ell)$ with $1 \le \ell \le m+1$.
The same is true for $S(z_{i,e})$, $S(y_{i,b}')$ and $S(z_{i,e}')$.
}
\end{remark}

\begin{lemma}
\label{lem:case1-2}
Let $r=[m,n]$, where $m, n \ge 3$ are integers.
Under Hypothesis~C, the following hold for every $i$.
\begin{enumerate}[\indent \rm (1)]
\item
$S(z_iy_{i+1})$ is not equal to $(m-1)$ nor $ (m)$.

\item $S(z_iy_{i+1})$ is not equal to $(m-1, m)$,
$(m, m-1)$ nor $(m, m)$.
\end{enumerate}
\end{lemma}

\begin{proof}
(1) Suppose on the contrary that $S(z_1y_2)=(m)$.
(The other case is treated analogously.)
Since $|z_1|,|y_2|>0$, we have
$|z_1|=d$ and $|y_2|=m-d$ for some $1 \le d \le m-1$.
Here, if $J=M$ (see Figure~\ref{fig.II-lemma-1-1}(a)),
then $CS(\phi(\delta^{-1}))=CS(s')$ contains both a term $m$ and
a term $m+2$, contradicting \cite[Lemma~3.8]{lee_sakuma_2}.
On the other hand, let $J \subsetneq M$.
Since none of $S(\phi(e_j'))$ contains $S_2$ in its interior
(see \cite[Lemma~3.1(2)]{lee_sakuma_3}),
wee see that the initial vertex of $e_2'$ lies in
the (central) segment of $\partial D_1^-$ corresponding to $S_2=((n-1) \langle m \rangle)$
and that the terminal vertex of $e_3'$ lies in
the (central) segment of $\partial D_2^-$ corresponding to $S_2=((n-1) \langle m \rangle)$.
Thus we see that $CS(\phi(\partial D_1'))=CS(r)$
contains a term of the form
$m+2+c$ with $c \in \ZZ_+ \cup \{0\}$,
as illustrated in Figure~\ref{fig.II-lemma-1-1}(b), a contradiction.

\begin{figure}[h]
\includegraphics{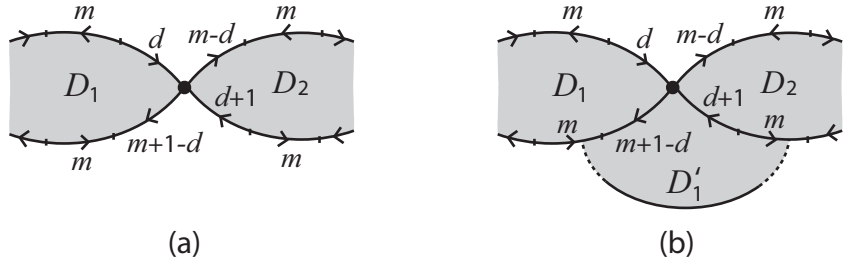}
\caption{
Lemma~\ref{lem:case1-2}(1) where $S(z_1y_2)=(d+(m-d))$}
\label{fig.II-lemma-1-1}
\end{figure}

(2) Suppose on the contrary that $S(z_1y_2)=(m-1, m)$.
(The other cases are treated similarly.)
Then $|z_1|=m-1$ and $|y_2|=m$.
Here, if $J=M$ (see Figure~\ref{fig.II-lemma-1-2}(a)),
then $CS(\phi (\delta^{-1}))=CS(s')$
contains both a term $1$ and a term $m$,
contradicting \cite[Lemma~3.8]{lee_sakuma_2}.
On the other hand, if $J \subsetneq M$
(see Figure~\ref{fig.II-lemma-1-2}(b)),
then $|\phi(e_2')|=2$ and $|\phi(e_3')|=1$,
for otherwise
we see, by using \cite[Lemma~3.1(2)]{lee_sakuma_3}
as in the proof of (1), that a subsequence of the form $(\ell_1, 1, \ell_2)$
or of the form $(\ell_1, 2, \ell_2)$ with $\ell_1, \ell_2 \in \ZZ_+$
would occur in $S(\phi(e_2'e_3'))$ which implies that
$CS(\phi(\partial D_1'))=CS(r)$ would contain a term $1$
or a term $2$, a contradiction.

Assuming that $e_2', e_3', {e_3''}^{-1}, {e_2''}^{-1}$ is
a boundary cycle of $D_1'$,
this implies that $\phi(e_2''e_3'')$ contains a subword $w$
such that $S(w)$ contains
$(S_1,S_2)=(m+1, (n-1) \langle m \rangle)$
as a proper initial subsequence or
$(S_2,S_1)=((n-1) \langle m \rangle,m+1)$ as
a proper terminal subsequence.
But this is impossible by \cite[Corollary~3.25(2)]{lee_sakuma_2}.
\end{proof}

\begin{figure}[h]
\includegraphics{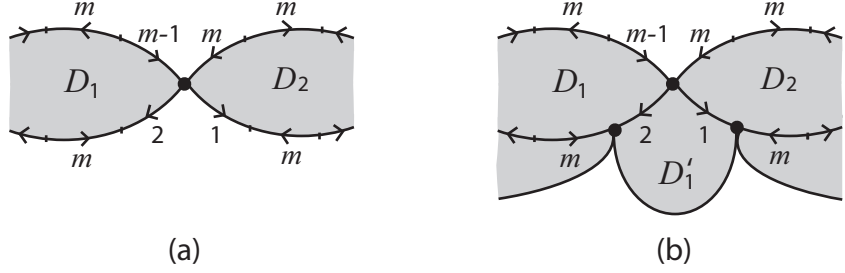}
\caption{
Lemma~\ref{lem:case1-2}(2) where $S(z_1y_2)=(m-1, m)$}
\label{fig.II-lemma-1-2}
\end{figure}

\begin{lemma}
\label{lem:case3-2}
Let $r=[m,1,n]$, where $m, n \ge 3$ are integers.
Under Hypothesis~C, the following hold for every $i$.
\begin{enumerate}[\indent \rm (1)]
\item $S(z_{i,e}y_{i+1,b})$ is not equal to $(m-1)$ nor $(m)$.

\item $S(z_{i,e}y_{i+1,b})$ is not equal to $(m-1, m)$, $(m, m-1)$ nor
$(m, m)$.

\item $S(z_{i,e}y_{i+1,b})$ is not equal to $(m-1, m-1)$.

\item $S(z_{i,e}y_{i+1,b})$ is not equal to $(m, m+1)$ nor $(m+1, m)$.
\end{enumerate}
\end{lemma}

\begin{proof}
The proofs of (1) and (2) are
parallel to those of Lemma~\ref{lem:case1-2}(1) and (2),
respectively.

(3) Suppose on the contrary that $S(z_{1,e}y_{2,b})=(m-1, m-1)$.
Then $|z_{1,e}|=|y_{2,b}|=m-1$.
Moreover, we must have $z_1=z_{1,e}$ and $y_2=y_{2,b}$.
To see this, suppose $z_1\ne z_{1,e}$. (The other case is treated similarly.)
Then, since $S_1=(n\langle m+1\rangle)$,
we see that $S(w_1z_1)$ is of the form $(S_2,m+1,*)$, where $*$ is nonempty.
So, $CS(\phi(\alpha))=CS(s)$ contains a term $m+1$.
This is a contradiction to \cite[Lemma~3.8]{lee_sakuma_2},
because $CS(\phi(\alpha))=CS(s)$ also contains $m-1$ by the assumption.

Here, if $J=M$ (see Figure~\ref{fig.II-lemma-3-3}(a)),
then $CS(\phi(\delta^{-1}))=CS(s')$ involves both a term $2$ and a term $m+1$.
Since $m+1 \ge 4$, we obtain a contradiction to \cite[Lemma~3.8]{lee_sakuma_2}.
On the other hand, if $J \subsetneq M$,
then by \cite[Lemma~3.1(2)]{lee_sakuma_3}
the initial vertex of $e_2'$ lies in
the segment of $\partial D_1^-$ corresponding to $S_2=(m)$
and that the terminal vertex of $e_3'$ lies in
the segment of $\partial D_2^-$ corresponding to $S_2=(m)$.
This implies that a subsequence of the form $(\ell_1, 2, \ell_2)$
with $\ell_1, \ell_2 \in \ZZ_+$
occurs in $S(\phi(e_2'e_3'))$ (see Figure~\ref{fig.II-lemma-3-3}(b)),
so in $CS(\phi(\partial D_1'))=CS(r)$ contains a term $2$,
a contradiction.

\begin{figure}[h]
\includegraphics{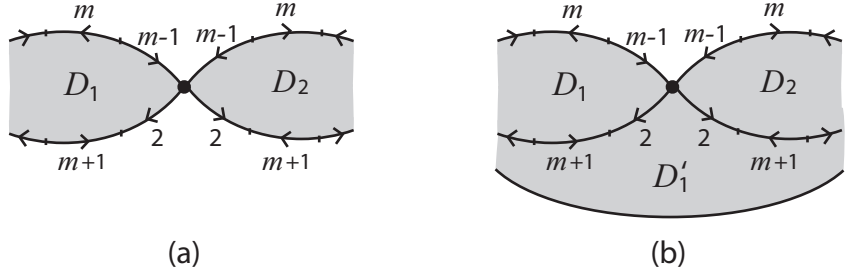}
\caption{
Lemma~\ref{lem:case3-2}(3) where $S(z_{1,e}y_{2,b})=(m-1, m-1)$}
\label{fig.II-lemma-3-3}
\end{figure}

(4) Suppose on the contrary that $S(z_{1,e}y_{2,b})=(m, m+1)$.
(The other case is treated analogously.)
Then $|z_{1,e}|=m$ and $|y_{2,b}|=m+1$.
Here, if $J=M$ (see Figure~\ref{fig.II-lemma-3-4}),
then $CS(\phi(\delta^{-1}))=CS(s')$ involves both a term $m$ and a term $m+2$,
contradicting \cite[Lemma~3.8]{lee_sakuma_2}.
On the other hand, if $J \subsetneq M$,
then we see, by using \cite[Lemma~3.1(2)]{lee_sakuma_3}
as in the proof of (3),
that a term of the form $m+2+c$ with $c \in \ZZ_+ \cup \{0\}$
occurs in $S(\phi(e_2'e_3'))$,
so in $CS(\phi(\partial D_1'))=CS(r)$ contains a term $m+2+c$,
a contradiction.
\end{proof}

\begin{figure}[h]
\includegraphics{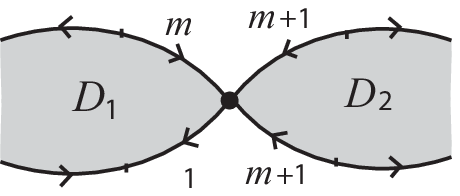}
\caption{
Lemma~\ref{lem:case3-2}(4) where $S(z_{1,e}y_{2,b})=(m, m+1)$}
\label{fig.II-lemma-3-4}
\end{figure}

\section{Proof of Main Theorem~\ref{main_theorem} for $r=[m,n]$
with $m, n\ge 3$}
\label{sec:proof_of_main_theorem_1}

Suppose $r=[m,n]$, where $m, n \ge 3$ are integers.
For two distinct elements $s, s' \in I_1(r) \cup I_2(r)$,
suppose on the contrary that the unoriented loops
$\alpha_s$ and $\alpha_{s'}$ are homotopic in $S^3-K(r)$,
namely suppose that Hypothesis~A holds.
We will derive a contradiction in each case to consider.
By \cite[Lemma~3.3]{lee_sakuma_3}, there are two big cases
to consider.

\medskip
\noindent
{\bf Case 1.} {\it Hypothesis~B holds.}
\medskip

By Corollary~\ref{cor:case1-1},
$CS(s)$ consists of $m$ and $m+1$.
Without loss of generality,
we may assume that a term $m$ occurs in $S(z_1y_2)$.
There are three possibilities:
\begin{enumerate}[\indent \rm (i)]
\item $S(z_1y_2)$ consists of only $m$, where
$S(z_1)=(n_1\langle m \rangle)$,
$S(y_2)=(n_2\langle m \rangle)$,
and $S(z_1y_2)=((n_1+n_2)\langle m \rangle)$
with $n_1,n_2\in\ZZ_+\cup\{0\}$;

\item $S(z_1y_2)$ consists of only $m$, where
$S(z_1)=(n_1\langle m \rangle,d)$,
$S(y_2)=(m-d,n_2\langle m \rangle)$,
and $S(z_1y_2)=((n_1+n_2+1)\langle m \rangle)$
with $n_1,n_2\in\ZZ_+\cup\{0\}$ and $d \in \{1, 2, \dots, m-1\}$;

\item $S(z_1y_2)$ consists of $m$ and $m+1$.
\end{enumerate}

First assume that (i) occurs.
Then $S(z_1'y_2')=((n_1'+n_2')\langle m \rangle)$
where $n_1'=(n-1)-n_1$ and $n_2'=(n-1)-n_2$.
So $n_1'+n_2'=2(n-1)-(n_1+n_2)$ and hence
either $S(z_1y_2)$ or $S(z_1'y_2')$ contains
$n-1$ consecutive $m$'s.
If $J=M$, then this implies that either $s \notin I_1(r) \cup I_2(r)$
or $s' \notin I_1(r) \cup I_2(r)$ by \cite[Proposition~3.19(1)]{lee_sakuma_2},
contradicting the hypothesis of the theorem.
On the other hand, if $J \subsetneq M$, then
the above observation implies that
either $S(z_1y_2)$
contains $n-1$ consecutive $m$'s and so $s \notin I_1(r) \cup I_2(r)$,
or otherwise $S(z_1'y_2')$ contains $n$ consecutive $m$'s.
The former case is impossible by the assumption.
In the latter case,
we see, by an argument using \cite[Lemma~3.1(1)]{lee_sakuma_3}
as in the last step of the proof of Lemma~\ref{lem:case1-1(b)}(1),
that a subsequence of the form $(\ell_1, n \langle m \rangle, \ell_2)$
with $\ell_1, \ell_2 \in \ZZ_+$
occurs in $S(\phi(e_2'e_3'))$, so in $CS(\phi(\partial D_1'))=CS(r)$,
a contradiction.
Next assume that (ii) occurs.
Then $S(z_1'y_2')=((n_1'+n_2'+1)\langle m \rangle)$,
where $n_1'=(n-2)-n_1$ and $n_2'=(n-2)-n_2$.
By using the identity
$n_1'+n_2'+1=2(n-1)-(n_1+n_2+1)$,
this case is treated as in the case when (i) occurs.
Finally assume that (iii) occurs.
Then we must have $S(z_{1, e}y_{2, b})=(m+1)$,
but this is a contradiction to Lemma~\ref{lem:case1-1(b)}(1).

\medskip
\noindent
{\bf Case 2.} {\it Hypothesis~C holds.}
\medskip

By Remark~\ref{rem:(2)holds}(1),
$CS(s)$ contains $((n-1)\langle m\rangle)$ as a proper subsequence.
Thus \cite[Lemma~3.8]{lee_sakuma_2} implies that
$CS(s)$ consists of $\{m-1,m\}$ or $\{m,m+1\}$.
Moreover, since $n\ge 3$, this together with \cite[Lemma~3.8]{lee_sakuma_2}
implies that $CS(s)$ does not contain
$(m-1,m-1)$ nor $(m+1,m+1)$ as a subsequence.

\medskip
\noindent
{\bf Case 2.a.}
{\it $CS(s)$ consists of $m-1$ and $m$.}
\medskip

In this case, a term $m-1$ should occur in $S(z_iy_{i+1})$ for some $i$.
Since $CS(s)$ does not contain $(m-1,m-1)$
and since each of $S(z_i)$ and $S(y_i)$ has length $1$,
we see that $S(z_iy_{i+1})$ is equal to $(m-1)$,
$(m,m-1)$ or $(m-1,m)$.
But, this is impossible by Lemma~\ref{lem:case1-2}.

\medskip
\noindent
{\bf Case 2.b.}
 {\it $CS(s)$ consists of $m$ and $m+1$.}
\medskip

In this case, $CS(s)$ contains both $S_1=(m+1)$ and
$S_2=((n-1)\langle m\rangle)$ as subsequences.
Hence, by \cite[Proposition~3.19(1)]{lee_sakuma_2},
$s \notin I_1(r) \cup I_2(r)$, contradicting the hypothesis of the theorem.
\qed

\section{Proof of Main Theorem~\ref{main_theorem} for $r=[m,1,n]$ with
$m, n \ge 3$}
\label{sec:proof_of_main_theorem_3}

Suppose $r=[m,1,n]$, where $m, n \ge 3$ are integers.
For two distinct elements $s, s' \in I_1(r) \cup I_2(r)$,
suppose on the contrary that the unoriented loops
$\alpha_s$ and $\alpha_{s'}$ are homotopic in $S^3-K(r)$,
namely suppose that Hypothesis~A holds.
We will derive a contradiction in each case to consider.
By \cite[Lemma~3.3]{lee_sakuma_3}, there are two big cases
to consider.

\medskip
\noindent
{\bf Case 1.} {\it Hypothesis~B holds.}
\medskip

By Corollary \ref{cor:case2-1}, we have the following two subcases.

\medskip
\noindent
{\bf Case 1.a.} {\it $CS(s)=\lp m+1, m+1 \rp$.}
\medskip

Since $\phi(\alpha)$ involves a subword $w_i$ whose $S$-sequence is
$(n \langle m+1 \rangle)$ and since $n\ge 3$, this is impossible.

\medskip
\noindent
{\bf Case 1.b.} {\it $CS(s)$ consists of $m$ and $m+1$.}
\medskip

In this case, we see that $CS(s)$ contains a subsequence
$S_1=(n \langle m+1 \rangle)$.
Since it also contains a subsequence $S_2=(m)$ by the assumption,
we see by \cite[Proposition~3.19(1)]{lee_sakuma_2} that
$s \notin I_1(r) \cup I_2(r)$, contradicting the hypothesis of the theorem.

\medskip
\noindent
{\bf Case 2.} {\it Hypothesis~C holds.}
\medskip

By Remark~\ref{rem:(2)holds}(2),
$CS(s)$ contains the term $m$.
So, by \cite[Lemma~3.8]{lee_sakuma_2},
Case~2 is reduced to the following three subcases:
$CS(s)=\lp m, m \rp$,
$CS(s)$ consists of $\{m-1, m\}$, or
$CS(s)$ consists of $\{m, m+1\}$.

\medskip
\noindent
{\bf Case 2.a.} {\it $CS(s)=\lp m, m \rp$.}
\medskip

In this case, there is only one possibility:
$J$ consists of one $2$-cell, namely
$CS(\phi(\alpha))=CS(\phi(\partial D_1^+))=\lp S(z_1y_1), S(w_1) \rp= \lp m, m \rp$.
Then $S(z_1y_1)=S(z_{1,e}y_{1,b})=(m)$, contradicting Lemma~\ref{lem:case3-2}(1).

\medskip
\noindent
{\bf Case 2.b.} {\it $CS(s)$ consists of $m-1$ and $m$.}
\medskip

In this case, a term $m-1$ must occur in $S(z_{i,e}y_{i+1,b})$ for some $i$.
But by Lemma~\ref{lem:case3-2}(1), (2) and (3), this is impossible.

\medskip
\noindent
{\bf Case 2.c.} {\it $CS(s)$ consists of $m$ and $m+1$.}
\medskip

Without loss of generality,
we may assume that $m+1$ occurs in $S(z_1y_2)$.
There are three possibilities:
\begin{enumerate}[\indent \rm (i)]
\item $S(z_1y_2)$ consists of only $m+1$, where
$S(z_1)=(n_1\langle m+1 \rangle)$,
$S(y_2)=(n_2\langle m+1 \rangle)$,
and $S(z_1y_2)=((n_1+n_2)\langle m+1 \rangle)$
with $n_1,n_2\in\ZZ_+\cup\{0\}$;

\item $S(z_1y_2)$ consists of only $m+1$, where
$S(z_1)=(n_1\langle m+1 \rangle,d)$,
$S(y_2)=(m+1-d,n_2\langle m+1 \rangle)$,
and $S(z_1y_2)=((n_1+n_2+1)\langle m+1 \rangle)$
with $n_1,n_2\in\ZZ_+\cup\{0\}$ and $d \in \{1, \dots, m\}$;

\item $S(z_1y_2)$ consists of $m$ and $m+1$.
\end{enumerate}
However, by an argument as in Case~1(i) and (ii)
in Section~\ref{sec:proof_of_main_theorem_1},
we can see that neither (i) nor (ii) can happen.
(In the above, we need to appeal to \cite[Lemma~3.1(2)]{lee_sakuma_3}
instead of \cite[Lemma~3.1(1)]{lee_sakuma_3}.)
So only (iii) can occur. But then a term $m$ must occur in $S(z_{1,e}y_{2,b})$,
a contradiction to Lemma~\ref{lem:case3-2}(1), (2) and (4).
\qed

\section{Preliminary results for the general cases}
\label{sec:technical_lemmas_general}

The remainder of this paper is devoted to the proof
of Main Theorem~\ref{main_theorem}
when $r$ is {\it general},
namely either $r=[m,m_2, \dots, m_k]$, where $m \ge 2$, $m_2 \ge 2$ and $k \ge 3$,
or $r=[m, 1, m_3, \dots, m_k]$, where $m \ge 2$ and $k \ge 4$.
In the remainder of this paper, we assume that $r$ is general.

\begin{remark}
\label{rem:general_decomposition}
{\rm
(1) Let $r=[m,m_2, \dots, m_k]$, where $m \ge 2$, $m_2 \ge 2$ and $k \ge 3$.
Then $\tilde{r}=[m_2-1, m_3, \dots, m_k]$ by \cite[Lemma~3.11]{lee_sakuma_2},
and so, by \cite[Proposition~3.12]{lee_sakuma_2},
$CS(\tilde{r})=\lp T_1,T_2,T_1,T_2 \rp$, where
$T_1$ begins and ends with $m_2$ and $T_2$ begins and ends with $m_2-1$.
Thus by \cite[Lemma~3.16(4)]{lee_sakuma_2}, we obtain that $CS(r)=\lp S_1, S_2, S_1, S_2 \rp$,
where $S_1$ begins and ends with $(m+1, (m_2-1) \langle m \rangle, m+1)$, and
$S_2$ begins and ends with $(m_2 \langle m \rangle)$.

(2) Let $r=[m,1,m_3, \dots, m_k]$, where $m \ge 2$ and $k \ge 4$.
Then $\tilde{r}=[m_3, \dots, m_k]$ by \cite[Lemma~3.11]{lee_sakuma_2},
and so, by \cite[Proposition~3.12]{lee_sakuma_2},
$CS(\tilde{r})=\lp T_1,T_2,T_1,T_2 \rp$, where
$T_1$ begins and ends with $m_3+1$ and $T_2$ begins and ends with $m_3$.
Thus by \cite[Lemma~3.16(2)]{lee_sakuma_2}, we obtain that $CS(r)=\lp S_1, S_2, S_1, S_2 \rp$,
where $S_1$ begins and ends with $((m_3+1) \langle m+1 \rangle)$,
and $S_2$ begins and ends with $(m, m_3 \langle m+1 \rangle, m)$.
}
\end{remark}

The aim of this section is to prove the following proposition.

\begin{proposition}
\label{prop:general_m_and_m+1}
Let $r=[m,m_2, \dots, m_k]$, where $m \ge 2$, $m_2 \ge 2$ and $k \ge 3$,
or $r=[m, 1, m_3, \dots, m_k]$, where $m \ge 2$ and $k \ge 4$.
Then under Hypothesis~A,
both $CS(s)$ and $CS(s')$ consist of $m$ and $m+1$.
Moreover $CS(s)$ contains $S_1$ or $S_2$ as a subsequence
accordingly as Hypothesis~B
or Hypothesis~C is satisfied.
\end{proposition}

In the following, we prove the proposition only for $CS(s)$.
By applying the same argument to the annular diagram
reversing the outer and inner boundaries, we see that
the assertion also holds for $CS(s')$.

\subsection{The case when Hypothesis~B holds}

In this subsection, we study the case when
Hypothesis~B holds.

Suppose $r=[m,m_2, \dots, m_k]$, where $m \ge 2$, $m_2 \ge 2$ and $k \ge 3$.
Then $S_1$ begins and ends with $(m+1, (m_2-1) \langle m \rangle, m+1)$
by Remark~\ref{rem:general_decomposition}(1).
Thus Hypothesis~B implies that $CS(\phi(\alpha))=CS(s)$ contains a term
$m$ and a term $m+1+\ell$ for some $\ell \ge 0$.
Hence $CS(s)$ must consist of $m$ and $m+1$ by \cite[Lemma~3.8]{lee_sakuma_2}.
Moreover, since $S_1$ begins and ends with $m+1$,
this observation together with Hypothesis~B implies that
$CS(s)$ contains $S_1$ as a subsequence.
Thus Proposition~\ref{prop:general_m_and_m+1} holds in this case.

Suppose $r=[m, 1, m_3, \dots, m_k]$, where $m \ge 2$ and $k \ge 4$.
The following lemma is needed for the proof of Lemma~\ref{lem:general_prelim_3-1},
by which we prove Proposition~\ref{prop:general_m_and_m+1}
for this type of $r$.

\begin{lemma}
\label{lem:general_prelim_3-1(b)}
Let $r=[2, 1, m_3, \dots, m_k]$, where $k \ge 4$.
Under Hypothesis~B, the following hold.
\begin{enumerate}[\indent \rm (1)]
\item $S(z_{i, e}y_{i+1, b}) \neq (4)$.

\item $S(w_{i, e}z_iy_{i+1,b}) \neq (4)$ and $S(z_{i-1, e}y_iw_{i, b}) \neq (4)$.

\item $S(w_{i, e}z_iy_{i+1,b}) \neq (5)$ and $S(z_{i-1, e}y_iw_{i, b}) \neq (5)$.

\item $S(w_{i, e}z_iy_{i+1}w_{i+1, b}) \neq (6)$.
\end{enumerate}
\end{lemma}

\begin{proof}
(1) Suppose on the contrary that $S(z_{1, e}y_{2, b})=(4)$.
Then we must have that $z_1=z_{1, e}$ and $y_2=y_{2, b}$.
To see this, suppose that $z_1\ne z_{1, e}$. (The other case is treated similarly.)
Then, since $S_2$ begins with $(2, m_3 \langle 3 \rangle, 2)$,
we see that $S(w_1z_1)$ is of the form $(S_1,2,*)$, where $*$ is nonempty.
So, $CS(\phi(\alpha))=CS(s)$ contains a term $2$.
This is a contradiction to \cite[Lemma~3.8]{lee_sakuma_2}, because
$CS(\phi(\alpha))$ also contains a term $4$ by the assumption.
Hence, $S(z_1)=S(z_{1, e})=(2)$ and $S(y_2)=S(y_{2, b})=(2)$.
Since $S_2$ begins and ends with $(2, m_3 \langle 3 \rangle, 2)$,
$S(z_1'y_2')$ contains a subsequence of the form $(\ell_1, 6, \ell_2)$
with $\ell_1, \ell_2 \in \ZZ_+$
(see Figure~\ref{fig.Lemma7_3_1}(a)).
So if $J=M$, then $CS(\phi(\delta^{-1}))=CS(s')$ contains both a term $2$
and a term $6$, contradicting \cite[Lemma~3.8]{lee_sakuma_2}.
Suppose $J \subsetneq M$.
Then the above fact together with \cite[Lemma~3.1(1)]{lee_sakuma_3}
implies that $S(\phi(e_2'e_3'))$ contains the subsequence $(\ell_1, 6, \ell_2)$
(see Figure~\ref{fig.Lemma7_3_1}(b)).
Thus, a term $6$ occurs in $CS(\phi(\partial D_1'))=CS(r)$, which is a contradiction.

\begin{figure}[h]
\includegraphics{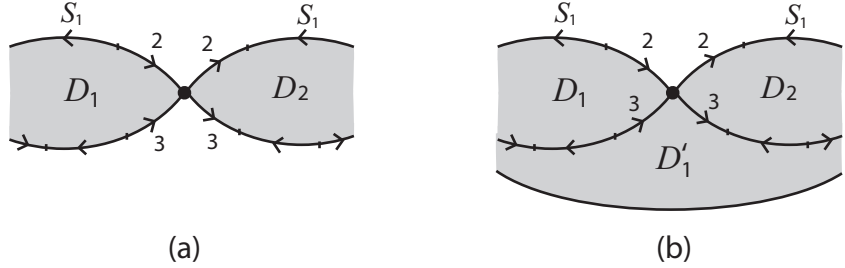}
\caption{
Lemma~\ref{lem:general_prelim_3-1(b)}(1) where $S(z_{1,e}y_{2,b})=(2+2)$}
\label{fig.Lemma7_3_1}
\end{figure}

(2) Suppose on the contrary that $S(w_{1, e}z_1y_{2,b})=(4)$.
(The other case is treated similarly.)
Then $|z_1|=0$ and $|y_{2, b}|=1$, because $|w_{1, e}|=3$.
Furthermore $y_2=y_{2, b}$ as in the proof of (1).
By using the fact that
$S_2$ begins and ends with $(2, m_3 \langle 3 \rangle, 2)$
and \cite[Lemma~3.1(1)]{lee_sakuma_3},
we see that
$S(\phi(e_2'e_3'))$ contains a subsequence of the form $(\ell_1, 2, 1, \ell_2)$
with $\ell_1, \ell_2 \in \ZZ_+$ (see Figure~\ref{fig.Lemma7_3_2}(a)).
So, if $J=M$, then $CS(\phi(\delta^{-1}))=CS(s')$ contains both
a term $1$ and a term $3$, contradicting \cite[Lemma~3.8]{lee_sakuma_2}.
On the other hand, if $J \subsetneq M$,
then a term $1$ occurs in $CS(\phi(\partial D_1'))=CS(r)$,
which is a contradiction.

(3) Suppose on the contrary that $S(w_{1, e}z_1y_{2,b})=(5)$.
(The other case is treated similarly.)
Then $|z_1|=0$ and $|y_{2, b}|=2$.
Furthermore $y_2=y_{2, b}$ as in the proof of (1).
So $S(y_2w_2)$ begins with a subsequence $(2, (m_3+1) \langle 3 \rangle)$,
which implies that $CS(\phi(\alpha))=CS(s)$ has a term $3$.
Since $CS(s)$ has a term $5$ by assumption,
this is a contradiction to \cite[Lemma~3.8]{lee_sakuma_2}.

(4) Suppose on the contrary that $S(w_{1, e}z_1y_2w_{2, b})=(6)$.
Then $|z_1|=|y_2|=0$ and $S(w_{1, e}w_{2, b})=(6)$.
By using the fact that
$S_2$ begins and ends with $(2, m_3 \langle 3 \rangle, 2)$
and \cite[Lemma~3.1(1)]{lee_sakuma_3}, we see that
$S(\phi(e_2'e_3'))$ contains a subsequence of the form $(\ell_1, 4, \ell_2)$
with $\ell_1, \ell_2 \in \ZZ_+$
(see Figure~\ref{fig.Lemma7_3_2}(b)).
So, if $J=M$, then $CS(\phi(\delta^{-1}))=CS(s')$ contains both
a term $2$ and a term $4$, contradicting \cite[Lemma~3.8]{lee_sakuma_2}.
On the other hand, if $J \subsetneq M$,
then a term $4$ occurs in $CS(\phi(\partial D_1'))=CS(r)$,
which is a contradiction.
\end{proof}

\begin{figure}[h]
\includegraphics{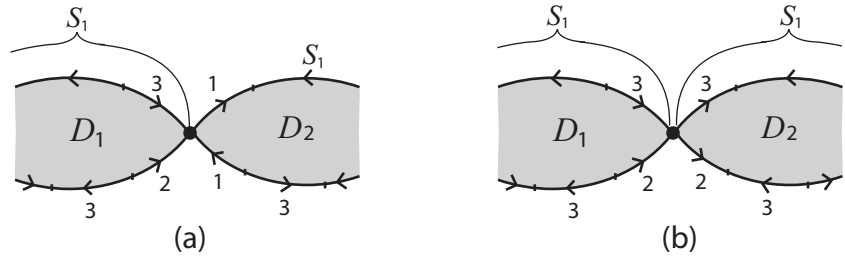}
\caption{
(a) Lemma~\ref{lem:general_prelim_3-1(b)}(2) where $S(w_{1, e}z_1y_{2,b})=(3+0+1)$, and
(b) Lemma~\ref{lem:general_prelim_3-1(b)}(4) where $S(w_{1, e}z_1y_2w_{2, b})=(3+0+3)$}
\label{fig.Lemma7_3_2}
\end{figure}

\begin{lemma}
\label{lem:general_prelim_3-1(b)_addition}
Let $r=[m, 1, m_3, \dots, m_k]$, where $m \ge 3$ and $k \ge 4$.
Under Hypothesis~B, the following hold.
\begin{enumerate}[\indent \rm (1)]
\item $S(z_{i, e}y_{i+1, b}) \neq (m+d)$ for any $d \in {\mathbb Z}_+$ with
$2 \le d \le m$.

\item $S(w_{i, e}z_iy_{i+1,b}) \neq (m+1+d)$ and $S(z_{i-1, e}y_iw_{i, b}) \neq (m+1+d)$
for any $d \in {\mathbb Z}_+$ with $1 \le d \le m$.

\item $S(w_{i, e}z_iy_{i+1}w_{i+1, b}) \neq (2m+2)$.
\end{enumerate}
\end{lemma}

\begin{proof}
(1) Suppose on the contrary that $S(z_{1, e}y_{2, b})=(m+d)$ with $2 \le d \le m$.
As in the proof of Lemma~\ref{lem:general_prelim_3-1(b)}(1),
we must have $z_1=z_{1, e}$ and $y_2=y_{2, b}$,
for otherwise $CS(\phi(\alpha))=CS(s)$ would contain both a term $m$
and a term $m+d$
contradicting \cite[Lemma~3.8]{lee_sakuma_2}.
So if $2 \le d\le m-1$,
then the proof is parallel to that of
Lemma~\ref{lem:case1-1(b)}(1).
Now let $d=m$.
Then $|z_1|=|y_2|=m$, and so $CS(\phi(\alpha))=CS(s)$ has a term $2m$.
Moreover, $S(y_2w_2)$ begins with $(m, (m_3+1) \langle m+1 \rangle)$,
and hence $CS(s)$ has a term $m+1$,
as shown in Figure~\ref{fig.lemma-3-2}.
But since $m \ge 3$, we have $2m > m+2$.
This gives a contradiction to \cite[Lemma~3.8]{lee_sakuma_2}.

\begin{figure}[h]
\includegraphics{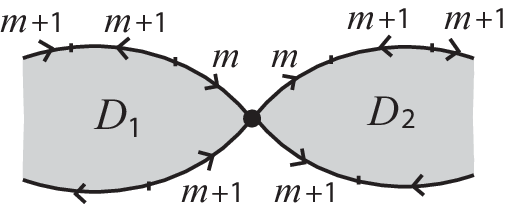}
\caption{
Lemma~\ref{lem:general_prelim_3-1(b)_addition}(1)
where $S(z_{1,e}y_{2,b})=(m+m)$}
\label{fig.lemma-3-2}
\end{figure}

(2) Suppose on the contrary that $S(w_{1, e}z_1y_{2,b})=(m+1+d)$.
(The other case is treated analogously.)
Then $|z_1|=0$ and $|y_{2, b}|=d$.
Furthermore $y_2=y_{2, b}$ as in the proof of (1).
So if $1 \le d\le m-1$, then the proof is parallel to that of
Lemma~\ref{lem:case1-1(b)}(2).
Now let $d=m$.
Then $|y_2|=m$ and $CS(s)$ has a term $2m+1$.
Also, as shown in the proof of (1),
$CS(s)$ has a term $m+1$.
But since $m \ge 3$, we have $2m+1 > m+2$.
This gives a contradiction to \cite[Lemma~3.8]{lee_sakuma_2}.

(3) Suppose on the contrary that $S(w_{1,e}z_1y_2w_{2,b})=(2m+2)$.
Then $|z_1|=|y_2|=0$ and $S(w_{1, e}w_{2, b})=(2m+2)$.
Assume first that $J=M$.
Then $CS(\phi(\delta^{-1}))=CS(s')$ contains both
a term $m+1$ and a term $2m$ as illustrated in Figure~\ref{fig.lemma-3-5}(a).
Since $m \ge 3$, this gives a contradiction to \cite[Lemma~3.8]{lee_sakuma_2}.
Assume next that $J \subsetneq M$.
Then by \cite[Lemma~3.1(1)]{lee_sakuma_3},
none of $S(\phi(e_{1}'))$ and $S(\phi(e_{2}'))$ contains
a subsequence $S_1$ which begins and ends with $((m_3+1) \langle m+1 \rangle)$.
This means that
the initial vertex of $e_2'$ lies in the interior
of the segment of $\partial D_1^-$ corresponding to $S_1$.
Similarly, the terminal vertex of $e_3'$
lies in the interior of the segment of $\partial D_2^-$
corresponding to $S_1$.
Hence, we see from Figure~\ref{fig.lemma-3-5}(b)
that $S(\phi(e_2'e_3'))$ contains a subsequence of the form $(\ell_1, 2m, \ell_2)$
with $\ell_1, \ell_2 \in \ZZ_+$. This implies that a term $2m$ occurs
in $CS(\phi(\partial D_1'))=CS(r)$, which is a contradiction.
\end{proof}

\begin{figure}[h]
\includegraphics{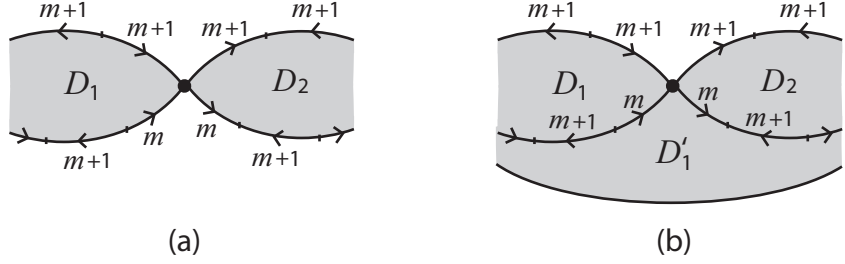}
\caption{
Lemma~\ref{lem:general_prelim_3-1(b)_addition}(3)
where $S(w_{1, e}z_1y_2w_{2, b})=((m+1)+0+0+(m+1))$}
\label{fig.lemma-3-5}
\end{figure}

\begin{lemma}
\label{lem:general_prelim_3-1}
Let $r=[m, 1, m_3, \dots, m_k]$, where $m \ge 2$ and $k \ge 4$.
Under Hypothesis~B,
no term of $CS(s)$ can be of the form $m+1+d$ with $d \in \ZZ_+$.
\end{lemma}

\begin{proof}
Suppose on the contrary that $CS(s)$ contains a term $m+1+d$.
Let $v$ be a subword of the cyclic word $(u_s)$
corresponding to a term $m+1+d$.
Without loss of generality, we may assume that
\begin{enumerate}[\indent \rm (i)]
\item $v$ contains $z_{1, e} y_{2, b}$ with $|z_{1, e}|, |y_{2, b}| \neq 0$;

\item $v$ contains $w_{1, e}y_{2, b}$ with $|y_{2, b}| \neq 0$;

\item $v$ contains $z_{0, e}w_{1, e}$ with $|z_{0, e}| \neq 0$; or

\item $v$ contains $w_{1, e}w_{2, b}$ with $|z_1|=|y_2|=0$.
\end{enumerate}
However, (i) is impossible by Lemma~\ref{lem:general_prelim_3-1(b)}(1) or
Lemma~\ref{lem:general_prelim_3-1(b)_addition}(1) accordingly as $m=2$ or $m \ge 3$.
Also (ii) and (iii) are impossible by Lemma~\ref{lem:general_prelim_3-1(b)}(2)--(3)
or Lemma~\ref{lem:general_prelim_3-1(b)_addition}(2) accordingly as $m=2$ or $m \ge 3$.
Finally (iv) is impossible by Lemma~\ref{lem:general_prelim_3-1(b)}(4)
or Lemma~\ref{lem:general_prelim_3-1(b)_addition}(3) accordingly as $m=2$ or $m \ge 3$.
\end{proof}

\begin{corollary}
\label{cor:general_prelim_3-1}
Let $r=[m, 1, m_3, \dots, m_k]$, where $m \ge 2$ and $k \ge 4$.
Under Hypothesis~B, the conclusion of Proposition~\ref{prop:general_m_and_m+1} holds.
\end{corollary}

\begin{proof}
By \cite[Lemma~3.8]{lee_sakuma_2}, either $CS(s)=\lp \ell, \ell \rp$
or $CS(s)$ consists of $\ell$ and $\ell+1$, for some $\ell \in \ZZ_+$.
Since $CS(\phi(\alpha))=CS(s)$ must
contain a term of the form $m+1+c$ with $c \in \ZZ_+ \cup \{0 \}$,
we have $\ell \ge m+1$ in the first case and $\ell \ge m$ in the second case.
First, if $CS(s)=\lp \ell, \ell \rp$,
then $\ell=m+1$ by Lemma~\ref{lem:general_prelim_3-1},
namely $CS(s)=\lp m+1, m+1 \rp$.
This happens only when $J$ consists of only one
$2$-cell with $CS(\phi(\partial D_1^+))=\lp S_1 \rp =\lp m+1, m+1\rp$.
Then $CS(\phi(\partial D_1^-))=\lp S_2, S_1, S_2\rp$,
and so $S(\phi(e_2'e_1'))$ contains a subsequence of the form $(\ell_1, S_2, S_2, \ell_2)$
with $\ell_1, \ell_2 \in \ZZ_+$.
So, if $J=M$, then $CS(\phi(\delta^{-1}))=CS(s')$ contains two consecutive $m$'s
and two consecutive $m+1$'s, contradicting \cite[Lemma~3.8]{lee_sakuma_2}.
On the other hand, if $J \subsetneq M$, then a subsequence $(S_2, S_2)$
occurs in $CS(\phi(\partial D_1'))=CS(r)$, a contradiction.
Thus $CS(s) \neq \lp \ell, \ell \rp$, and so $CS(s)$ consists of $\ell$ and $\ell+1$.
By Lemma~\ref{lem:general_prelim_3-1}, we have $\ell+1 \le m+1$,
so that $\ell=m$, as desired.
Hence, $CS(s)$ consists of $m$ and $m+1$.
As already observed in the beginning of this subsection,
this implies that $CS(s)$ contains $S_1$ as a subsequence.
\end{proof}

\subsection{The case when Hypothesis~C holds}

We next assume that Hypothesis~C holds.
In this case, $CS(s)$ contains $S_2$ as a subsequence.

Suppose $r=[m, 1, m_3, \dots, m_k]$, where $m \ge 2$ and $k \ge 4$.
Then $S_2$ begins and ends with $(m, m_3 \langle m+1 \rangle, m)$
by Remark~\ref{rem:general_decomposition}(2).
Hence $CS(\phi(\alpha))=CS(s)$ contains terms $m$ and $m+1$,
and therefore $CS(\phi(\alpha))=CS(s)$ consists of $m$ and $m+1$
by \cite[Lemma~3.8]{lee_sakuma_2}. So, Proposition~\ref{prop:general_m_and_m+1} holds in this case.

On the other hand, for $r=[m,m_2, \dots, m_k]$, where $m \ge 2$, $m_2 \ge 2$ and $k \ge 3$,
we prove the following lemmas, by which we prove Proposition~\ref{prop:general_m_and_m+1}
for this type of $r$.

\begin{lemma}
\label{lem:general_prelim_1-2_easy}
Let $r=[m,m_2, \dots, m_k]$, where $m \ge 2$, $m_2 \ge 2$ and $k \ge 3$.
Under Hypothesis~C, $CS(s)$ consists of at least three terms including $m$.
\end{lemma}

\begin{proof}
The assertion immediately follows from the fact that
$CS(\phi(\alpha))=CS(s)$ properly contains $S_2$ which begins and ends with
$(m_2 \langle m \rangle)$ (see Remark~\ref{rem:general_decomposition}(1)).
\end{proof}

\begin{lemma}
\label{lem:general_prelim_1-2}
Let $r=[m,m_2, \dots, m_k]$, where $m \ge 2$, $m_2 \ge 2$ and $k \ge 3$.
Under Hypothesis~C, the following hold for every $i$.
\begin{enumerate}[\indent \rm (1)]
\item $S(z_{i, e}y_{i+1, b}) \neq (m-1)$.

\item $S(z_{i, e}y_{i+1, b}) \neq (m-1, m-1)$.

\item $S(z_{i, e}y_{i+1, b}) \neq (m-1, m)$ and $S(z_{i, e}y_{i+1, b}) \neq (m, m-1)$.
\end{enumerate}
\end{lemma}

\begin{proof}
(1) Suppose on the contrary that $S(z_{1, e}y_{2, b})=(m-1)$.
Then we have $z_1=z_{1, e}$ and $y_2=y_{2, b}$, for otherwise
$CS(\phi(\alpha))=CS(s)$ contains both a term $m-1$ and a term $m+1$,
contradicting \cite[Lemma~3.8]{lee_sakuma_2}.
By using \cite[Lemma~3.1(2)]{lee_sakuma_3} as in the proof of
Lemma~\ref{lem:case1-2}(1), we see that
$S(\phi(e_2'e_3'))$ contains a subsequence of the form
$(\ell_1, m+3, \ell_2)$ with $\ell_1, \ell_2 \in \ZZ_+$.
So, if $J=M$, then $CS(\phi(\delta^{-1}))=CS(s')$ contains both
a term $m$ and a term $m+3$, contradicting \cite[Lemma~3.8]{lee_sakuma_2}.
On the other hand, if $J \subsetneq M$, then a term $m+3$ occurs in
$CS(\phi(\partial D_1'))=CS(r)$, a contradiction.

(2) Suppose on the contrary that $S(z_{1, e}y_{2, b})=(m-1, m-1)$.
Then $CS(\phi(\alpha))=CS(s)$ involves two consecutive $m-1$'s.
On the other hand, since $CS(s)$ contains $S_2$,
which begins and ends with $(m_2 \langle m \rangle)$,
we see that $CS(s)$ also contains two consecutive $m$'s.
This is a contradiction to \cite[Lemma~3.8]{lee_sakuma_2}.

(3) Suppose on the contrary that $S(z_{1, e}y_{2, b})=(m-1, m)$.
(The other case is treated similarly.)
As in the proof of (1),
$|z_{1, e}|=m-1$, $|y_{2, b}|=m$,
$z_1=z_{1, e}$ and $y_2=y_{2, b}$.
By using \cite[Lemma~3.1(2)]{lee_sakuma_3} as in the proof of
Lemma~\ref{lem:case1-2}(2), we see that
$S(\phi(e_2'e_3'))$ contains a subsequence of the form
$(\ell_1, 2, 1, \ell_2)$ with $\ell_1, \ell_2 \in \ZZ_+$.
Hence, if $J=M$, then $CS(\phi(\delta^{-1}))=CS(s')$ contains both
a term $1$ and a term $m+1$.
Since $m+1 \ge 3$, we have a contradiction to \cite[Lemma~3.8]{lee_sakuma_2}.
On the other hand, if $J \subsetneq M$, then a term
$1$ occurs in $CS(\phi(\partial D_1'))=CS(r)$, a contradiction.
\end{proof}

\begin{corollary}
\label{cor:general_prelim_1-2}
Let $r=[m,m_2, \dots, m_k]$, where $m \ge 2$, $m_2 \ge 2$ and $k \ge 3$.
Under Hypothesis~C, the conclusion of Proposition~\ref{prop:general_m_and_m+1} holds.
\end{corollary}

\begin{proof}
By Lemma~\ref{lem:general_prelim_1-2_easy}, $CS(s)$ consists of at least three terms including $m$.
Also, Lemma~\ref{lem:general_prelim_1-2} shows that no term of $CS(s)$ can be $m-1$.
Hence by \cite[Lemma~3.8]{lee_sakuma_2}, $CS(s)$ must consist of $m$ and $m+1$.
Moreover, we have already observed that $CS(s)$ contains $S_2$ as a subsequence.
\end{proof}

Thus we have proved Proposition~\ref{prop:general_m_and_m+1}.

\section{Transformation of diagrams for the general cases}
\label{sec:transformation}

We first introduce a concept for a vertex of $M$ to be
converging, diverging or mixing.
To this end, we subdivide the edges of $M$ so that the label of
any oriented edge in the subdivision has length $1$.
We call each of the edges in the subdivision a {\it unit segment}
in order to distinguish them from the edges in the original $M$.

\begin{definition}
\label{def:vertex_type}
{\rm
(1) A vertex $x$ in $M$ is said to be {\it converging} (resp., {\it diverging})
if the set of labels of incoming
unit segments of $x$ is $\{a, b\}$ (resp., $\{a^{-1}, b^{-1}\}$).
See Figure~\ref{fig.converging} and its caption for description.

(2) A vertex $x$ in $M$ is said to be {\it mixing}
if the set of labels of incoming
unit segments of $x$ is $\{a, a^{-1},b, b^{-1}\}$.
See Figure~\ref{fig.impossible_tsequence} and its caption for description.
}
\end{definition}

\begin{figure}[h]
\includegraphics{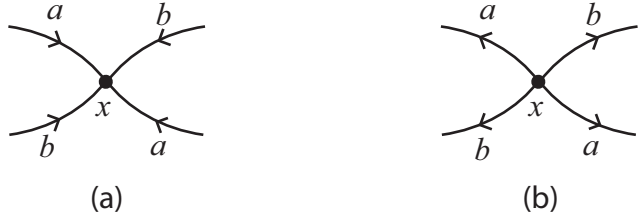}
\caption{
Orient each of the unit segment so that
the associated label is equal to $a$ or $b$.
Then a vertex $x$ is (a) converging (resp., (b) diverging)
if all unit segments incident on $x$
are oriented so that they are converging into $x$
(resp., diverging from $x$).}
\label{fig.converging}
\end{figure}

\begin{figure}[h]
\includegraphics{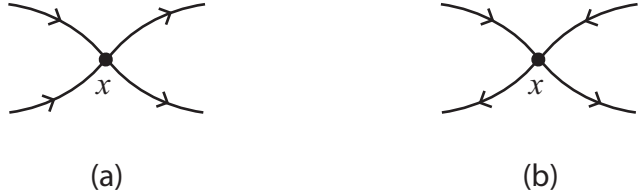}
\caption{
A vertex $x$ is mixing if it looks like as in the above when
we orient the segments as in Convention~\ref{con:figure}, namely,
the change of directions of consecutive arrowheads
represents the change from positive (negative, resp.) words
to negative (positive, resp.) words.}
\label{fig.impossible_tsequence}
\end{figure}

As we declared in the previous section,
we assume that $r$ is general.
The main purpose of this section
is to show that, under Hypothesis~A, we may modify $M$ so that
every vertex $x$ in $M$ with $\deg_M (x)=4$
is either converging or diverging
(see Corollary~\ref{cor:general_critical(aa)}).

\subsection{The case when vertices lie in the outer boundary layer}

We first treat a vertex $x$ in the outer boundary layer $J$ with $\deg_J (x)=4$.
To do this, we need several lemmas.

\begin{lemma}
\label{lem:converging}
Assume that $r$ is general.
Under Hypothesis~B,
suppose that the vertex between $D_i$ and $D_{i+1}$
is either converging or diverging.
Then none of the following occurs.
\begin{enumerate}[\indent \rm (1)]
\item $S(\phi(\partial D_i^+))$ ends with $S_1$.

\item $S(\phi(\partial D_{i+1}^+))$ begins with $S_1$.

\item $S(\phi(\partial D_i^-))$ ends with $S_1$.

\item $S(\phi(\partial D_{i+1}^-))$ begins with $S_1$.
\end{enumerate}
\end{lemma}

\begin{proof}
We may assume that $i=1$ and that the vertex is diverging.
Suppose on the contrary that (1) occurs,
namely suppose that $S(\phi(\partial D_1^+))$ ends with $S_1$.
Then $S(\phi(\partial D_1^-))$ ends with $(S_1,S_2)$.
Since the vertex is diverging,
$S(\phi(\partial D_1^-)\phi(\partial D_2^-))$ contains
a subsequence $(S_1,S_2,d)$ for some $d\in\ZZ_+$.
Suppose $J=M$.
Then $CS(\phi(\delta^{-1}))=CS(s')$ contains $S_2$.
(See Figure~\ref{fig.S_1_occurs}(a),
keeping in mind Convention~\ref{con:figure}.)
Moreover, the subsequence $S_1$ of $S(\phi(\partial D_1^-))$
also forms a subsequence of $CS(\phi(\delta^{-1}))=CS(s')$,
because $CS(s')$ consists of $m$ and $m+1$
(Proposition~\ref{prop:general_m_and_m+1}) and
$S_1$ begins and ends with $m+1$.
Thus $CS(s')$ contains both $S_1$ and $S_2$,
yielding that $s' \notin I_1(r) \cup I_2(r)$
by \cite[Proposition~3.19(1)]{lee_sakuma_2}, a contradiction.
Suppose $J \subsetneq M$.
Then we see in the following that
the two $2$-cells $D_1$ and $D_1'$ in Figure~\ref{fig.S_1_occurs}(b)
form a reducible pair, contradicting that
$M$ is a reduced diagram.
To see this, let $\gamma$ (resp., $\gamma'$)
be the boundary cycle of $D_1$ (resp., $D_1'$) which goes around
the boundary in clockwise (resp., counter-clockwise)
direction starting from the vertex between $D_1$ and $D_2$.
Let $\gamma_0$ be the common initial segment of $\gamma$ and
$\gamma'$ such that $S(\phi(\gamma_0))=S_2$.
Then we see by using \cite[Lemma~3.1(1)]{lee_sakuma_3}
that the terminal point of $\gamma_0$
($=$ the initial point of $e_2'$) is contained in the interior
of the segment of $\partial D_1^-$ corresponding to
the segment $S_1$ (see Figure~\ref{fig.S_1_occurs}(b)).
Thus, in both reduced words $\phi(\gamma)$ and $\phi(\gamma')$,
there are ``sign changes'' just before and after $\phi(\gamma_0)$.
Hence $S(\phi(\gamma_0))=S_2$ is a subsequence of
both $CS(\phi(\gamma))=CS(r)$ and $CS(\phi(\gamma'))=CS(r)$.
By \cite[Proposition~3.12]{lee_sakuma_2}(2)
and the fact that
$\phi(\gamma)$ and $\phi(\gamma')$ shares $\phi(\gamma_0)$
as a common initial word,
we must have $\phi(\gamma) \equiv \phi(\gamma')$; so
$D_1$ and $D_1'$ form a reducible pair.
Thus we have proved that (1) does not occur.
Similarly, we can prove that (2) does not occur.

\begin{figure}[h]
\includegraphics{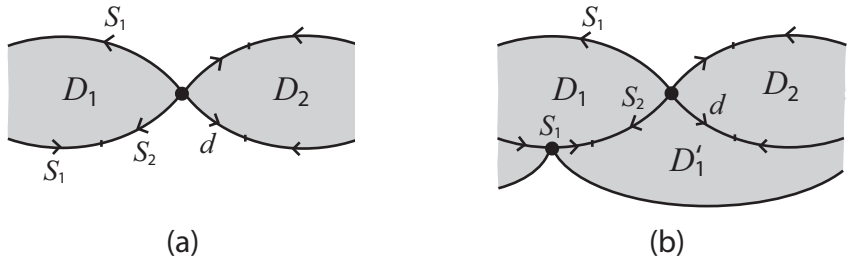}
\caption{
Lemma~\ref{lem:converging}(1)
where $S(\phi(\partial D_1^+))$ ends with $S_1$}
\label{fig.S_1_occurs}
\end{figure}

Suppose on the contrary that (3) occurs, i.e.,
$S(\phi(\partial D_1^-))$ ends with $S_1$.
Then $S(\phi(\partial D_1^+))$ ends with $(S_1,S_2)$.
Since the vertex is diverging, this subsequence $S_2$ of $S(\phi(\partial D_1^+))$
is also a subsequence of $CS(\phi(\alpha))=CS(s)$.
Moreover, we can see as in the proof of (1) that
the subsequence $S_1$ of $S(\phi(\partial D_1^+))$
forms a subsequence of $CS(\phi(\alpha))$.
Thus $CS(\phi(\alpha))$ contains both $S_1$ and $S_2$ as subsequences
and so $s \notin I_1(r) \cup I_2(r)$
by \cite[Proposition~3.19(1)]{lee_sakuma_2}, a contradiction.
The assertion (4) is proved similarly.
\end{proof}

\begin{lemma}
\label{lem:general_critical(a)}
Assume that $r$ is general.
Under Hypothesis~B, we may assume that the following hold
for every face $D_i$ of $J$.
\begin{enumerate}[\indent \rm (1)]
\item $S(\phi(\partial D_i^+))$ contains a subsequence of the form
$(\ell, S_1, \ell')$ with $\ell, \ell' \in \ZZ_+$.

\item $S(\phi(\partial D_i^-))$ contains a subsequence of the form
$(\ell, S_1, \ell')$ with $\ell, \ell' \in \ZZ_+$.
\end{enumerate}
To be precise, we can modify the reduced annular diagram $M$
into a reduced annular diagram $M'$
keeping the outer and inner boundary labels unchanged
so that every $2$-cell of the outer boundary layer of $M'$
satisfies the above conditions.
\end{lemma}

\begin{proof}
Suppose that the assertion does not hold.
Then one of the four (prohibited) conditions
in Lemma~\ref{lem:converging} holds.
In particular,
the vertex between $D_i$ and $D_{i+1}$
is not converging nor diverging.

Suppose that condition (1) in Lemma~\ref{lem:converging}
occurs. Then we may assume that
$S(\phi(\partial D_1^+))$ ends with $S_1$ and so $|z_1|=0$.
Then $S(z_1')=S_2$, and so $S(z_{1, e}')=(m)$.
Hence the sequence $S(z_{1, e}'y_2'w_2')$ begins with
a subsequence of the form either $(m, d)$ or $(m+d)$,
where $d \in \ZZ_+$.
Suppose that $S(z_{1, e}'y_2'w_2')$ begins with
a subsequence of the form $(m, d)$.
Then the vertex between $D_1$ and $D_2$ is either converging or diverging,
a contradiction.
(In fact, since $CS(s)=CS(\phi(\alpha))$ consists of $m$ and $m+1$
by Proposition~\ref{prop:general_m_and_m+1}
and since $S(\phi(\partial D_1^+))$ ends with $m+1$,
there is a ``sign change'' between
$\phi(\partial D_1^+)$ and $\phi(\partial D_2^+)$.
Thus $J$ is locally
as illustrated in Figure~\ref{fig.S_1_occurs}(a)
up to simultaneous change of the edge orientations.)
So $S(z_{1, e}'y_2'w_2')$ must begin with
a subsequence of the form $(m+d)$. Since $S(\phi(\partial D_1^+))$
ends with a term $m+1$ and since $CS(\phi(\alpha))=CS(s)$ consists of
$m$ and $m+1$, the only possibility is that $d=1$ and
$S(w_{1, e} y_{2, b})=(m+1, m)$.
Then, as illustrated in Figure~\ref{fig.S_1_occurs(b)},
we may transform $M$
so that $S(\phi(\partial D_1^+))$ ends with $(S_1, m)$.
To be precise, we cut $J$ at the black vertex in the left figure in
Figure~\ref{fig.S_1_occurs(b)} and then
identify the two white vertices.
The resulting diagram is illustrated
in the right figure in Figure~\ref{fig.S_1_occurs(b)},
where the black vertex is the image of the white vertices.
It should be noted that the boundary labels of $J$ are unchanged
by this operation
and so we can glue $M-J$ to this new $J$.
This modification does not change the boundary labels of $M$
and the new vertex of $J$ is converging or diverging.
Hence we see by Lemma~\ref{lem:converging}
that none of conditions (1)--(4) occurs at the vertex
between $D_1$ and $D_2$ in this new diagram.
Thus we have shown that
condition (1) in Lemma~\ref{lem:converging} may be assumed not to occur.
Similarly, we can show that
condition (2) in Lemma~\ref{lem:converging} may be assumed not to occur.

\begin{figure}[h]
\includegraphics{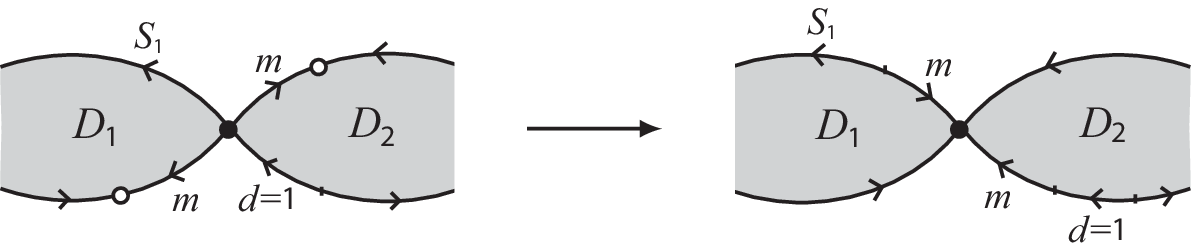}
\caption{
Lemma~\ref{lem:general_critical(a)}
where $S(z_{1, e}'y_2'w_2')$ begins with
$(m+d)$}
\label{fig.S_1_occurs(b)}
\end{figure}

Suppose that condition (3) in Lemma~\ref{lem:converging} occurs.
Then we may assume $S(\phi(\partial D_1^-))$ ends with $S_1$
and so $|z_1'|=0$.
Then $S(z_1)=S_2$, and so $S(z_{1, e})=(m)$.
Hence the sequence $S(z_{1, e}y_2w_2)$ begins with
a subsequence of the form either $(m, d)$ or $(m+d)$,
where $d \in \ZZ_+$.
If $S(z_{1, e}y_2w_2)$ begins with
a subsequence of the form $(m, d)$,
then we can see as in the previous case that
the vertex between $D_1$ and $D_2$ is
converging or diverging, a contradiction.
So $S(z_{1, e}y_2w_2)$ must begin with
a subsequence of the form $(m+d)$. Since $S(\phi(\partial D_1^+))$
ends with a term $m$ and since $CS(\phi(\alpha))=CS(s)$ consists of
$m$ and $m+1$, the only possibility is that $d=1$ and
$S(w_{1, e}' y_{2, b}')=(m+1, m)$.
Then, as illustrated in Figure~\ref{fig.S_1_occurs(c)},
we may transform $M$
so that $S(\phi(\partial D_1^-))$ ends with $(S_1, m)$.
Since the new vertex is either converging or diverging,
we see by Lemma~\ref{lem:converging}
that none of conditions (1)--(4) occurs at the common vertex
of $D_1$ and $D_2$ in this new diagram.
Thus we have shown that condition (3)
in Lemma~\ref{lem:converging} may be assumed not to occur.
Similarly, we can show that
condition (4) in Lemma~\ref{lem:converging} may be assumed not to occur.
This completes the proof of Lemma~\ref{lem:general_critical(a)}.
\end{proof}

\begin{figure}[h]
\includegraphics{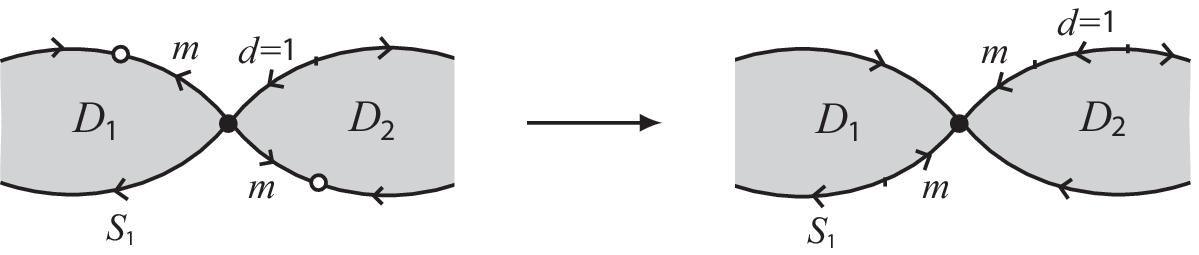}
\caption{
Lemma~\ref{lem:general_critical(a)}
where $S(z_{1, e}y_2w_2)$ begins with
$(m+d)$}
\label{fig.S_1_occurs(c)}
\end{figure}

In the remainder of this section,
when we assume Hypothesis~B,
we always assume that
the two conditions in Lemma~\ref{lem:general_critical(a)}
hold, namely the words $y_i$, $z_i$, $y_i'$ and $z_i'$
which appear in the expressions of $\phi(\partial D_i^+)$ and $\phi(\partial D_i^-)$
are nonempty. Note that when we assume Hypothesis~C,
the same conditions always hold by the hypothesis.

\begin{lemma}
\label{lem:general_critical(b)}
Assume that $r$ is general.
Under Hypothesis~B or Hypothesis~C,
the following hold for every $i$.
\begin{enumerate}[\indent \rm (1)]
\item If $S(z_{i, e} z_{i, e}'^{-1})=S(y_{i+1, b}'^{-1}y_{i+1, b})=(m)$,
then $S(z_{i, e}y_{i+1, b}) \neq (m+1)$.

\item If $S(z_{i, e} z_{i, e}'^{-1})=S(y_{i+1, b}'^{-1}y_{i+1, b})=(m+1)$,
then $S(z_{i, e}y_{i+1, b}) \neq (m)$.

\item If $S(z_{i, e} z_{i, e}'^{-1})=(m)$ and $S(y_{i+1, b}'^{-1}y_{i+1, b})=(m, m)$,
then $S(z_{i, e}y_{i+1, b}) \neq (m+1)$.

\item If $S(z_{i, e} z_{i, e}'^{-1})=(m)$ and $S(y_{i+1, b}'^{-1}y_{i+1, b})=(m+1, m)$,
then $S(z_{i, e}y_{i+1, b}) \neq (m+1)$.

\item If $S(z_{i, e} z_{i, e}'^{-1})=S(y_{i+1, b}'^{-1}y_{i+1, b})=(m+1)$,
then $S(z_{i, e}y_{i+1, b}) \neq (m, m)$.

\item If $S(z_{i, e} z_{i, e}'^{-1})=(m, m+1)$ and
$S(y_{i+1, b}'^{-1}y_{i+1, b})=(m+1)$,
then $S(z_{i, e}y_{i+1, b}) \neq (m, m)$.

\item If $S(z_{i, e} z_{i, e}'^{-1})=(m+1, m+1)$ and
$S(y_{i+1, b}'^{-1}y_{i+1, b})=(m+1)$,
then $S(z_{i, e}y_{i+1, b}) \neq (m+1, m)$.

\end{enumerate}
\end{lemma}

\begin{proof}
(1) Let $S(z_{1, e} z_{1, e}'^{-1})=S(y_{2, b}'^{-1}y_{2, b})=(m)$.
Suppose on the contrary that $S(z_{1, e}y_{2, b})=(m+1)$
(see Figure~\ref{fig.general_critical_b1}(a)).
Then $S(\phi(e_2'e_3'))$ contains a subsequence of the form
$(\ell_1, m-1, \ell_2)$ with $\ell_1, \ell_2 \in \ZZ_+$.
Here, if $J=M$, then $CS(\phi(\delta^{-1}))=CS(s')$ contains a term $m-1$,
a contradiction to Proposition~\ref{prop:general_m_and_m+1}.
On the other hand, if $J \subsetneq M$,
then a term $m-1$ occurs in $CS(\phi(\partial D_1'))=CS(r)$,
a contradiction
(cf. Proof of Lemma~\ref{lem:case1-1(b)}(1) for the case $J \subsetneq M$).

(2) Let $S(z_{1, e} z_{1, e}'^{-1})=S(y_{2, b}'^{-1}y_{2, b})=(m+1)$.
Suppose on the contrary that $S(z_{1, e}y_{2, b})=(m)$
(see Figure~\ref{fig.general_critical_b1}(b)).
Then $S(\phi(e_2'e_3'))$ contains a term $m+2$.
Here, if $J=M$, then $CS(\phi(\delta^{-1}))=CS(s')$
contains a term $m+2$,
a contradiction to Proposition~\ref{prop:general_m_and_m+1}.
On the other hand, if $J \subsetneq M$,
then a term of the form $m+2+c$ with $c \in \ZZ_+ \cup \{0\}$
occurs in $CS(\phi(\partial D_1'))=CS(r)$,
a contradiction.

\begin{figure}[h]
\includegraphics{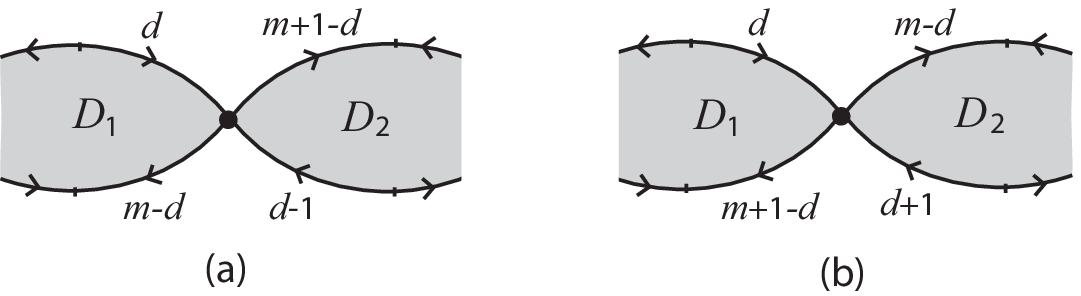}
\caption{
(a) Lemma~\ref{lem:general_critical(b)}(1),
and (b) Lemma~\ref{lem:general_critical(b)}(2)}
\label{fig.general_critical_b1}
\end{figure}

(3) Let $S(z_{1, e} z_{1, e}'^{-1})=(m)$ and $S(y_{2, b}'^{-1}y_{2, b})=(m, m)$.
Suppose on the contrary that $S(z_{1, e}y_{2, b})=(m+1)$
(see Figure~\ref{fig.general_critical_b3}(a)).
Then $S(\phi(e_2'e_3'))$ contains a subsequence of the form
$(\ell_1, m-1, \ell_2)$ with $\ell_1, \ell_2 \in \ZZ_+$.
Here, if $J=M$, then $CS(\phi(\delta^{-1}))=CS(s')$ contains
a term $m-1$, a contradiction to Proposition~\ref{prop:general_m_and_m+1}.
On the other hand, if $J \subsetneq M$,
then a term $m-1$ occurs in $CS(\phi(\partial D_1'))=CS(r)$,
a contradiction.

(4) Let $S(z_{1, e} z_{1, e}'^{-1})=(m)$ and $S(y_{2, b}'^{-1}y_{2, b})=(m+1, m)$.
Suppose on the contrary that $S(z_{1, e}y_{2, b})=(m+1)$
(see Figure~\ref{fig.general_critical_b3}(b)).
Then $S(\phi(e_2'e_3'))$ contains a subsequence of the form
$(\ell_1, m-1, \ell_2)$ with $\ell_1, \ell_2 \in \ZZ_+$.
So, arguing as in the proof of (3), we reach a contradiction.

\begin{figure}[h]
\includegraphics{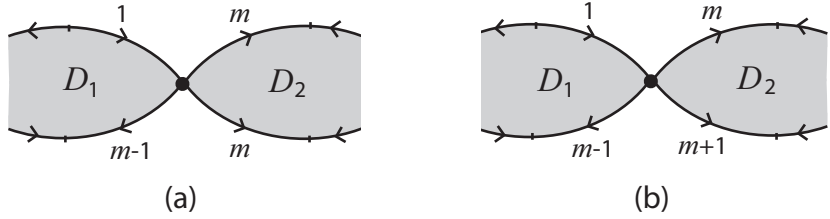}
\caption{
(a) Lemma~\ref{lem:general_critical(b)}(3),
and (b) Lemma~\ref{lem:general_critical(b)}(4)}
\label{fig.general_critical_b3}
\end{figure}

(5) Let $S(z_{1, e} z_{1, e}'^{-1})=S(y_{2, b}'^{-1}y_{2, b})=(m+1)$.
Suppose on the contrary that $S(z_{1, e}y_{2, b})=(m, m)$
(see Figure~\ref{fig.general_critical_b5}(a)).
Here, if $J=M$, then $CS(\phi(\delta^{-1}))=CS(s')$ contains a term $1$,
a contradiction to Proposition~\ref{prop:general_m_and_m+1}.
On the other hand, if $J \subsetneq M$,
then $S(\phi(e_2'e_3'))$ contains a subsequence of the form
$(\ell_1, 1, \ell_2)$ with $\ell_1, \ell_2 \in \ZZ_+$,
for otherwise $S(\phi(\partial {D_1'}^+))=S(\phi(e_2'e_3'))=(1, 1)$
which contains neither $S_1$ nor $(\ell, S_2, \ell')$ with $\ell, \ell' \in \ZZ_+$,
contradicting \cite[Lemma~3.2]{lee_sakuma_3}.
It the follows that a term $1$ occurs in $CS(\phi(\partial D_1'))=CS(r)$,
a contradiction.

(6) Let $S(z_{1, e} z_{1, e}'^{-1})=(m, m+1)$ and
$S(y_{2, b}'^{-1}y_{2, b})=(m+1)$.
Suppose on the contrary that $S(z_{1, e}y_{2, b})=(m, m)$
(see Figure~\ref{fig.general_critical_b5}(b)).
Then $S(\phi(e_2'e_3'))$ contains a term $m+2$.
Here, if $J=M$, then $CS(\phi(\delta^{-1}))=CS(s')$ contains a term $m+2$,
a contradiction to Proposition~\ref{prop:general_m_and_m+1}.
On the other hand, if $J \subsetneq M$,
then a term of the form $m+2+c$ with $c \in \ZZ_+ \cup \{0\}$
occurs in $CS(\phi(\partial D_1'))=CS(r)$,
a contradiction.

\begin{figure}[h]
\includegraphics{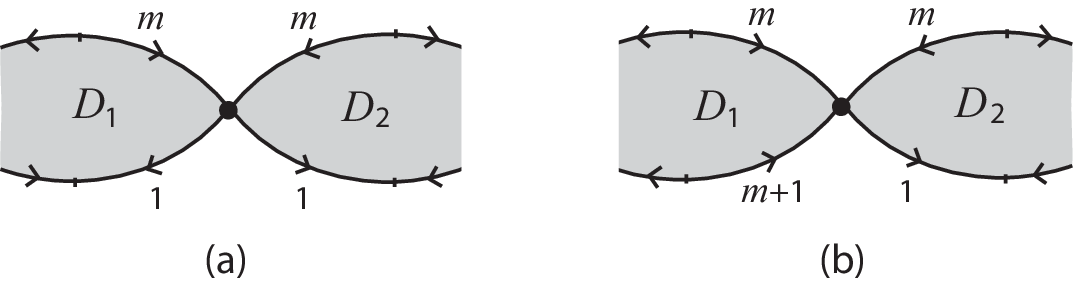}
\caption{
(a) Lemma~\ref{lem:general_critical(b)}(5),
and (b) Lemma~\ref{lem:general_critical(b)}(6)}
\label{fig.general_critical_b5}
\end{figure}

(7) Let $S(z_{1, e} z_{1, e}'^{-1})=(m+1, m+1)$ and
$S(y_{2, b}'^{-1}y_{2, b})=(m+1)$.
Suppose on the contrary that $S(z_{1, e}y_{2, b})=(m+1, m)$.
Then $S(\phi(e_2'e_3'))$ contains a term
$m+2$. So, arguing as in the proof of (6),
we reach a contradiction.
\end{proof}

We are now ready to prove the following.

\begin{proposition}
\label{prop:general_critical}
Assume that $r$ is general.
Under Hypothesis~B or Hypothesis~C,
we may assume that the following hold.
\begin{enumerate}[\indent \rm (1)]
\item For every face $D_i$ of $J$, the following hold.
\begin{enumerate}[\rm (a)]
\item If $S(\phi(\partial D_i^+))$ contains $S_1$,
then it contains a subsequence of the form
$(\ell, S_1, \ell')$ with $\ell, \ell' \in \ZZ_+$.

\item If $S(\phi(\partial D_i^-))$ contains $S_1$,
then it contains a subsequence of the form
$(\ell, S_1, \ell')$ with $\ell, \ell' \in \ZZ_+$.
\end{enumerate}

\item Every vertex $x$ in $J$ with $\deg_{J}(x)=4$
is either converging or diverging.
\end{enumerate}
\end{proposition}

\begin{proof}
By Lemma~\ref{lem:general_critical(a)},
we may assume that (1) is satisfied.
We prove that we can modify the annular diagram $M$ so that it satisfies (2).
Then the resulting annular diagram satisfies both (1) and (2),
because Lemma~\ref{lem:converging} guarantees that
if $M$ satisfies (2) then it also satisfies (1).

Since (1) is satisfied,
the subwords $y_i$ and $z_i$ of $\phi(\partial D_i^+)$
and the subwords $y_i'$ and $z_i'$ of $\phi(\partial D_i^-)$
are nonempty.
Suppose on the contrary that there is a vertex $x \in J$ with $\deg_J(x)=4$ such that
$x$ is neither converging nor diverging.
We may assume $x$ is the vertex between $D_1$ and $D_2$.
Then $x$ has one of the five types as depicted in Figure~\ref{fig.vertex_type},
where $c_i$ and $d_i$ ($i=1,2$) are positive integers,
up to simultaneous reversal of the edge orientations
and up to the reflection in the vertical edge passing through the vertex $x$.
To see this, let $L$ be the set of labels of incoming
unit segments of $x$, and
orient each of the unit segment so that
the associated label is equal to $a$ or $b$
as in Figure~\ref{fig.converging}.
If $L=\{a^{\pm 1}, b^{\pm 1}\}$, then
we obtain the situation (a) or (b) in Figure~\ref{fig.vertex_type}.
If $L$ consists of three elements,
then we may assume that $a$ and $a^{-1}$, respectively,
appear as the label of the upper left and lower right
incoming unit segments
and that $b$ or $b^{-1}$ does not belong to $L$.
Then we obtain the situation (c) or (d) in Figure~\ref{fig.vertex_type}.
If $L$ consists of two elements,
then we may assume both the upper left and lower right
incoming unit segments have label $a$,
and both the upper left and lower right
incoming unit segments have label $b^{-1}$,
because $x$ is not converging nor diverging.
In this case, we have the situation (e) in Figure~\ref{fig.vertex_type}.

\begin{figure}[h]
\includegraphics{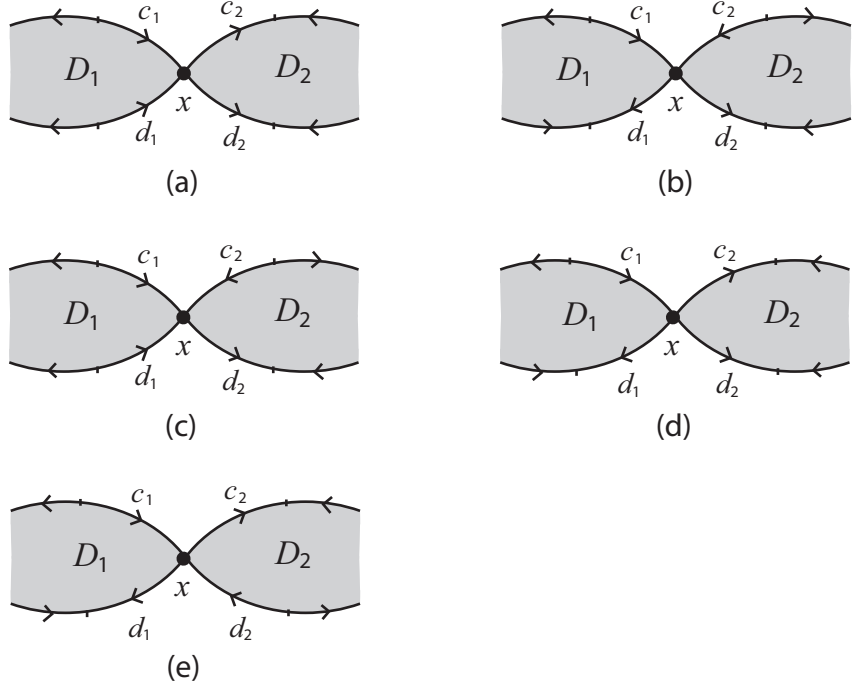}
\caption{
The five possible types of a vertex $x \in J$ with $\deg_J (x)=4$
such that $x$ is neither converging nor diverging}
\label{fig.vertex_type}
\end{figure}

Assume that $x$ is as depicted in Figure~\ref{fig.vertex_type}(a).
Then, for each $i=1,2$, $c_i$ is a term of $CS(\phi(\partial D_i))=CS(r)$
and so is equal to $m$ or $m+1$.
Hence the term, $c_1+c_2$, of $CS(\phi(\alpha))=CS(s)$ is at least $2m$.
This is a contradiction to Proposition~\ref{prop:general_m_and_m+1}.

Assume that $x$ is as depicted in Figure~\ref{fig.vertex_type}(b).
Then $(c_1,c_2)$ is a subsequence of $CS(\phi(\alpha))=CS(s)$.
Since $CS(\phi(\alpha))=CS(s)$ consists of $m$ and $m+1$
by Proposition~\ref{prop:general_m_and_m+1},
the only possibility is that
$c_1=c_2=m$ and $d_1=d_2=1$.
So, we must have
$S(z_{1, e} z_{1, e}'^{-1})=S(y_{2, b}'^{-1}y_{2, b})=(m+1)$
and $S(z_{1, e}y_{2, b}) = (m, m)$.
But this is impossible
by Lemma~\ref{lem:general_critical(b)}(5).

Assume that $x$ is as depicted in Figure~\ref{fig.vertex_type}(c).
Then $(c_1,c_2)$ is a subsequence of $CS(\phi(\alpha))=CS(s)$
and hence each of $c_1$ and $c_2$ is either $m$ or $m+1$.
Moreover, $c_2+d_2$ is a term of $CS(r)$ and hence
it is either $m$ or $m+1$.
So, we have the following four possibilities:
\begin{enumerate}[\indent \rm (i)]
\item $c_1=m$, $c_2=m$, $d_1=m$, $d_2=1$;

\item $c_1=m$, $c_2=m$, $d_1=m+1$, $d_2=1$;

\item $c_1=m+1$, $c_2=m$, $d_1=m$, $d_2=1$;

\item $c_1=m+1$, $c_2=m$, $d_1=m+1$, $d_2=1$.
\end{enumerate}
However, (ii) and (iv) are impossible by
Lemma~\ref{lem:general_critical(b)}(6)
and Lemma~\ref{lem:general_critical(b)}(7), respectively.
If (i) or (iii) happens, then $c_2=d_1$.
So we can transform $M$ so that $x$ is diverging
as in Figure~\ref{fig.transformation_2}
(cf. the argument in the proof of Lemma~\ref{lem:general_critical(a)}
appealing to Figure~\ref{fig.S_1_occurs(b)}).

\begin{figure}[h]
\includegraphics{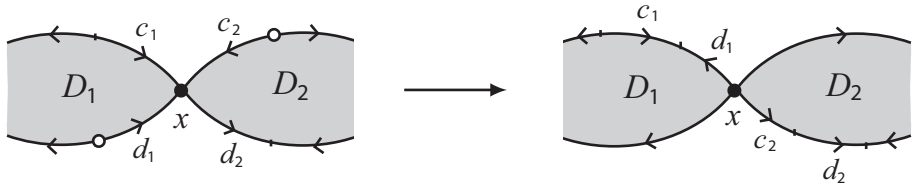}
\caption{
The transformation of Figure~\ref{fig.vertex_type}(c) when
$c_2=d_1$ so that $x$ is diverging}
\label{fig.transformation_2}
\end{figure}

Assume that $x$ is as depicted in Figure~\ref{fig.vertex_type}(d).
Then $c_1+c_2$ is a term of $CS(\phi(\alpha))=CS(s)$
and hence it is either $m$ or $m+1$.
Moreover, $c_1+d_1$ is a term of $CS(r)$ and hence
it is either $m$ or $m+1$.
So, we have the following four possibilities:
\begin{enumerate}[\indent \rm (i)]
\item $c_1=1$, $c_2=m$, $d_1=m-1$, $d_2=m$;

\item $c_1=1$, $c_2=m$, $d_1=m-1$, $d_2=m+1$;

\item $c_1=1$, $c_2=m$, $d_1=m$, $d_2=m$;

\item $c_1=1$, $c_2=m$, $d_1=m$, $d_2=m+1$.
\end{enumerate}
However, (i) and (ii) are impossible by Lemma~\ref{lem:general_critical(b)}(3)
and Lemma~\ref{lem:general_critical(b)}(4), respectively.
If (iii) or (iv) happens, then $c_2=d_1$.
So we can transform $M$ so that $x$ is converging
as in Figure~\ref{fig.transformation_1}.

\begin{figure}[h]
\includegraphics{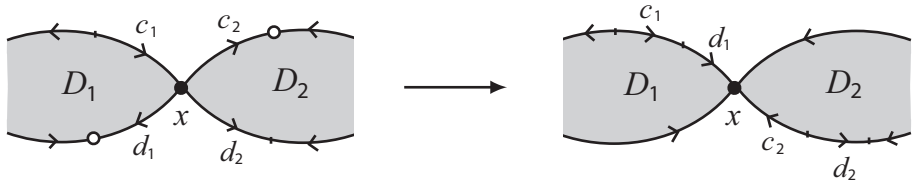}
\caption{
The transformation of Figure~\ref{fig.vertex_type}(d) when $c_2=d_1$
so that $x$ is converging}
\label{fig.transformation_1}
\end{figure}

Assume that $x$ is as depicted in Figure~\ref{fig.vertex_type}(e).
Then $c_1+c_2$ is a term of $CS(\phi(\alpha))=CS(s)$
and hence it is either $m$ or $m+1$.
Moreover, for each $i=1,2$, $c_i+d_i$ is a term of $CS(r)$ and hence
it is either $m$ or $m+1$.
So, we have the following eight possibilities:
\begin{enumerate}[\indent \rm (i)]
\item $c_1+c_2=m$, $c_1+d_1=m$, $c_2+d_2=m$;

\item $c_1+c_2=m$, $c_1+d_1=m$, $c_2+d_2=m+1$;

\item $c_1+c_2=m$, $c_1+d_1=m+1$, $c_2+d_2=m$;

\item $c_1+c_2=m$, $c_1+d_1=m+1$, $c_2+d_2=m+1$;

\item $c_1+c_2=m+1$, $c_1+d_1=m$, $c_2+d_2=m$;

\item $c_1+c_2=m+1$, $c_1+d_1=m$, $c_2+d_2=m+1$;

\item $c_1+c_2=m+1$, $c_1+d_1=m+1$, $c_2+d_2=m$;

\item $c_1+c_2=m+1$, $c_1+d_1=m+1$, $c_2+d_2=m+1$.
\end{enumerate}
However, (iv) and (v) are impossible by
Lemma~\ref{lem:general_critical(b)}(2)
and Lemma~\ref{lem:general_critical(b)}(1), respectively.
If (i), (ii), (vii) or (viii) happens, then $c_2=d_1$.
Thus, as illustrated in Figure~\ref{fig.transformation_1},
we may transform $M$ so that $x$ is converging
(cf. the argument in the proof of Lemma~\ref{lem:general_critical(a)}
appealing to Figure~\ref{fig.S_1_occurs(b)}).
If (iii) or (vi) happens, then $c_1=d_2$.
So we can transform $M$ so that $x$ is diverging
as in Figure~\ref{fig.transformation_2}.
\end{proof}

\subsection{The case when vertices lie in an arbitrary layer}

We next treat a general vertex $x$ in $M$ with $\deg_M (x)=4$
by using induction on the number of layers of $M$.
Note that the base step (i.e., the case when $M=J$) was already proved
in Proposition~\ref{prop:general_critical}(2).
We need several lemmas and new notations.

\begin{notation}
\label{not:layers}
{\rm
Under Hypothesis~A,
suppose that the number of layers of $M$ is $p+1$ with $p \ge 0$.
(Recall the characterization of $M$ in \cite[Theorem~4.11]{lee_sakuma_2}.)
For each $j=0, 1, \dots, p$, we define $J_j$ as follows:
$J_0=J$, and $J_{j}$ is the outer boundary layer of
$M-(J_0 \cup \cdots \cup J_{j-1})$ for $j \ge 1$.
Then $M=J_0 \cup J_1 \cup \cdots \cup J_p$.
We also define $\alpha_j$ and $\delta_j$ to be, respectively,
outer and inner boundary cycles of $J_j$ starting from a vertex
lying in both the outer and inner boundaries of $J_j$.
Let $\alpha_j=e_{j,1}, e_{j,2}, \dots, e_{j,2t}$ and
$\delta_j^{-1}=e_{j,1}', e_{j,2}', \dots, e_{j,2t}'$ be the
decompositions into oriented edges in $\partial J_j$.
Then clearly for each $i=1, \dots, t$,
there is a face $D_{j,i}$ of $J_j$ such that
$e_{j,2i-1}, e_{j,2i}, e_{j,2i}'^{-1}, e_{j,2i-1}'^{-1}$ are
consecutive edges in a boundary cycle of $D_{j,i}$.
We denote the path $e_{j,2i-1}, e_{j,2i}$ by $\partial D_{j,i}^+$ and
the path $e_{j,2i-1}', e_{j,2i}'$ by $\partial D_{j,i}^-$.
In particular, if $J_0 \cup \cdots \cup J_j \subsetneq M$,
then we may assume that $e_{j,2i}'$ and $e_{j,2i+1}'$ are
two consecutive edges in $\partial D_{j+1,i} \cap \delta_j^{-1}$.
}
\end{notation}

We can easily see that \cite[Lemmas~3.1 and 3.2]{lee_sakuma_3}
holds not only for the faces of $J_0=J$ but also
for every face of $M$.

\begin{lemma}
\label{lem:vertex_position:intermediate_layer}
Under Hypothesis~A and Notation~\ref{not:layers}, both of the following hold
for every face $D_{j,i}$ of $M$.
\begin{enumerate}[\indent \rm (1)]
\item None of $S(\phi(e_{j,2i-1}))$, $S(\phi(e_{j,2i}))$,
$S(\phi(e_{j,2i}'))$ and $S(\phi(e_{j,2i-1}'))$ contains
$S_1$ as a subsequence.

\item None of $S(\phi(e_{j,2i-1}))$, $S(\phi(e_{j,2i}))$,
$S(\phi(e_{j,2i}'))$ and $S(\phi(e_{j,2i-1}'))$ contains
a subsequence of the form $(\ell, S_2, \ell')$,
where $\ell, \ell' \in \ZZ_+$.
\end{enumerate}
\end{lemma}

\begin{lemma}
\label{lem:two_cases:intermediate_layer}
Under Hypothesis~A and Notation~\ref{not:layers}, only one of the following holds
for each face $D_{j,i}$ of $M$.
\begin{enumerate}[\indent \rm (1)]
\item Both $S(\phi(\partial D_{j,i}^+))$ and $S(\phi(\partial D_{j,i}^-))$
contain $S_1$ as its subsequence.

\item Both $S(\phi(\partial D_{j,i}^+))$ and $S(\phi(\partial D_{j,i}^-))$
contain a subsequence of the form $(\ell, S_2, \ell')$,
where $\ell, \ell' \in \ZZ_+$.
\end{enumerate}
\end{lemma}

In fact, these lemmas are proved by using \cite[Corollary~3.25]{lee_sakuma_2}
and the following facts:
\begin{enumerate}[\indent \rm (i)]
\item
The words $\phi(e_{j,i})$ and $\phi(e_{j,i}')$ are pieces.
(If $1\le j\le p-1$, this follows from the assumption that $M$ is a reduced annular diagram.
If $j=0$ or $p$, this follows from \cite[Convention~4.3]{lee_sakuma_2}.)
\item
The words $\phi(\partial D_{j,i}^+)$ and $\phi(\partial D_{j,i}^-)$
are not pieces.
(Otherwise, the cyclic word $(\phi(\partial D_{j,i}))=(u_r^{\pm1})$ becomes a product
of three pieces, a contradiction to \cite[Proposition~3.22]{lee_sakuma_2}.)
\end{enumerate}

\begin{notation}
\label{not:layers2}
{\rm
Under Hypothesis~A and Notation~\ref{not:layers},
Lemma~\ref{lem:two_cases:intermediate_layer} implies that
we may decompose the word $\phi(\alpha_j)$ into
\[
\phi(\alpha_j) \equiv y_{j,1} w_{j,1} z_{j,1} y_{j,2} w_{j,2} z_{j,2} \cdots y_{j,t} w_{j,t} z_{j,t},
\]
where $\phi(\partial D_{j,i}^+) \equiv \phi(e_{j,2i-1}e_{j,2i}) \equiv y_{j,i}w_{j,i}z_{j,i}$,
and where
either $S(w_{j,i})=S_1$ or both $S(w_{j,i})=S_2$ and $y_{j,i}, z_{j,i}$ are nonempty words.
We also have the decomposition of the word $\phi(\delta_j^{-1})$
as follows:
\[
\phi(\delta_j^{-1}) \equiv y_{j,1}' w_{j,1}' z_{j,1}' y_{j,2}' w_{j,2}' z_{j,2}' \cdots y_{j,t}' w_{j,t}' z_{j,t}',
\]
where $\phi(\partial D_{j,i}^-) \equiv \phi(e_{j,2i-1}'e_{j,2i}') \equiv y_{j,i}'w_{j,i}'z_{j,i}'$,
and where
either $S(w_{j,i}')=S_1$ or both $S(w_{j,i}')=S_2$ and $y_{j,i}', z_{j,i}'$ are nonempty words.
Here, the indices for the $2$-cells are considered modulo $t$,
and the indices for the edges are considered modulo $2t$.
}
\end{notation}

\begin{lemma}
\label{lem:converging2}
Assume that $r$ is general.
Under Hypothesis~A and Notation~\ref{not:layers},
suppose that the vertex between $D_{j,i}$ and $D_{j,i+1}$
is either converging or diverging.
Then none of the following occurs.
\begin{enumerate}[\indent \rm (1)]
\item $S(\phi(\partial D_{j,i}^+))$ ends with $S_1$.

\item $S(\phi(\partial D_{j,i+1}^+))$ begins with $S_1$.

\item $S(\phi(\partial D_{j,i}^-))$ ends with $S_1$.

\item $S(\phi(\partial D_{j,i+1}^-))$ begins with $S_1$.
\end{enumerate}
\end{lemma}

\begin{proof}
We may assume that $i=1$ and that the vertex is diverging.
If $j=0$, then the assertion is nothing other than
Lemma~\ref{lem:converging}.
If $j=p$, then the assertion is proved by applying
the proof of Lemma~\ref{lem:converging}
to the inner boundary layer of $M$.
So, we may assume $1\le j\le p-1$.
Suppose on the contrary that (1) occurs,
namely suppose that $S(\phi(\partial D_{j,1}^+))$ ends with $S_1$.
Then $S(\phi(\partial D_{j,1}^-))$ ends with $(S_1,S_2)$.
Since the vertex is diverging,
$S(\phi(\partial D_{j,1}^-)\phi(\partial D_{j,2}^-))$ contains
a subsequence $(S_1,S_2,d)$ for some $d\in\ZZ_+$.
Thus we obtain a situation as illustrated in
Figure~\ref{fig.S_1_occurs}(b),
where $D_{j,1}$, $D_{j,2}$ and $D_{j+1,1}$, respectively,
correspond to $D_1$, $D_2$ and $D_1'$ in the figure.
Then by the argument in the proof of Lemma~\ref{lem:converging}
for the case where (1) occurs and $J \subsetneq M$,
we see that the two $2$-cells $D_{j,1}$ and $D_{j+1,1}$
form a reducible pair, a contradiction.
(Here we use Lemma~\ref{lem:vertex_position:intermediate_layer}(1)
instead of \cite[Lemma~3.1(1)]{lee_sakuma_3}.)
Hence (1) cannot occur.
By a similar argument, we can see that
(2), (3) and (4) cannot occur.
\end{proof}

Now we are ready to prove the following
generalization of Proposition~\ref{prop:general_critical}.

\begin{proposition}
\label{prop:general_critical(aa)}
Assume that $r$ is general.
Under Hypothesis~A and Notation~\ref{not:layers},
we may assume that the following hold for every $j$.
\begin{enumerate}[\indent \rm (1)]
\item For every face $D_{j,i}$ of $J_j$, the following hold.
\begin{enumerate}[\rm (a)]
\item If $S(\phi(\partial D_{j,i}^+))$ contains $S_1$,
then it contains a subsequence of the form
$(\ell, S_1, \ell')$ with $\ell, \ell' \in \ZZ_+$.

\item If $S(\phi(\partial D_{j,i}^-))$ contains $S_1$,
then it contains a subsequence of the form
$(\ell, S_1, \ell')$ with $\ell, \ell' \in \ZZ_+$.
\end{enumerate}

\item Every vertex $x$ in $J_j$ with $\deg_{J_j}(x)=4$
is either converging or diverging.
\end{enumerate}
\end{proposition}

\begin{proof} We simultaneously prove
(1) and (2) by induction on $j \ge 0$.
The base step $j=0$ is already proved in
Proposition~\ref{prop:general_critical}. So fix $j \ge 1$.
By the inductive hypothesis,
$CS(\phi(\delta_{j-1}^{-1}))=CS(\phi(\alpha_j))$
consists of $m$ and $m+1$.

(1a)
Suppose on the contrary that $S(\phi(\partial D_{j,1}^+))$
ends with $S_1$ so that $|z_{j,1}|=0$
(see Notation~\ref{not:layers2}).
(The other case is treated similarly.)
Then $S(z_{j,1}')=S_2$ and so $S(z_{j,1,e}')=(m)$.
Note also that the assumption implies that
the vertex between $D_{j,1}$ and $D_{j,2}$ is not converging nor
diverging by Lemma~\ref{lem:converging2}.
Thus, arguing as in the proof of Lemma~\ref{lem:general_critical(a)}
when Lemma~\ref{lem:converging}(1) does not hold,
we see that $S(z_{j,1,e}'y_{j,2}'w_{j,2}')$ begins with
a subsequence of the form $(m+d)$ with $d \in \ZZ_+$.
Since $S(\phi(\partial D_{j,1}^+))=CS(\phi(\delta_{j-1}^{-1}))$
ends with a term $m+1$, and since $CS(\phi(\alpha_j))$ consists of
$m$ and $m+1$ by the inductive hypothesis,
the only possibility is that $d=1$ and
$S(w_{j,1,e} y_{j,2,b})=(m+1, m)$.
Then, as illustrated in Figure~\ref{fig.S_1_occurs(b)}, we may transform $M$
so that $S(\phi(\partial D_{j,1}^+))$ ends with $(S_1, m)$.
Since the new vertex of $M$ is either converging or diverging,
we see by Lemma~\ref{lem:converging2} that
none of the four (prohibited) conditions in the lemma holds.
Thus by repeating this argument at every degree $4$ vertex of $J_j$,
we obtain the desired result.

(1b) Suppose on the contrary that $S(\phi(\partial D_{j,1}^-))$
ends with $S_1$. (The other case is treated similarly.)
Then $|z_{j,1}'|=0$.
It follows that $S(z_{j,1})=S_2$, so that $S(z_{j,1,e})=(m)$.
Hence the sequence $S(z_{j,1,e}y_{j,2}w_{j,2})$ begins with
a subsequence of the form either $(m, d)$ or $(m+d)$,
where $d \in \ZZ_+$.
Here, if $S(z_{j,1,e}y_{j,2}w_{j,2})$ begins with
a subsequence of the form $(m, d)$,
then we see by an argument
as in the proof of Lemma \ref{lem:converging}(1) that
two $2$-cells $D_{j,1}$ and $D_{j-1,2}$ are a reducible pair,
contradicting that $M$ is a reduced diagram.
So $S(z_{j,1,e}y_{j,2}w_{j,2})$ must begin with
a subsequence of the form $(m+d)$.
Since $S(\phi(\partial D_{j,1}^+))$
ends with a term $m$, and since
$S(\phi(\partial D_{j,1}^+))=CS(\phi(\delta_{j-1}^{-1}))$
consists of $m$ and $m+1$ by the inductive hypothesis,
the only possibility is that $d=1$ and
$S(w_{j,1,e}' y_{j,2,b}')=(m+1, m)$.
Then, as illustrated in Figure~\ref{fig.S_1_occurs(c)},
we may transform $M$
so that $S(\phi(\partial D_{j,1}^-))$ ends with $(S_1, m)$.
Since the new vertex of $M$ is converging or diverging,
we see by Lemma~\ref{lem:converging2} that
none of the four (prohibited) conditions in the lemma holds.
Thus by repeating this argument at every degree $4$ vertex of $J_j$,
we obtain the desired result.

(2) As in the proof of Proposition~\ref{prop:general_critical},
we show that we can modify $M$ so that it satisfies (2).
Then it continues to satisfy (1) by Lemma~\ref{lem:converging2}.
To this end, note that
$CS(\phi(\alpha_j))$ consists of $m$ and $m+1$
by the inductive hypothesis.
By using this fact, we can see that
the statement of Lemma~\ref{lem:general_critical(b)} holds,
where $z_{i,e}$, $z_{i,e}'$, $y_{i+1,e}$ and $y_{i+1,e}'$ are replaced with
$z_{j,i,e}$, $z_{j,i,e}'$, $y_{j,i+1,e}$ and $y_{j,i+1,e}'$, respectively.
(The proof of Lemma~\ref{lem:general_critical(b)} works if
we appeal to the fact that
$CS(\phi(\alpha_j))$, instead of $CS(\phi(\alpha))$,
consists of $m$ and $m+1$.)
So we can follow the proof of Proposition~\ref{prop:general_critical}
and show that (2) holds.
\end{proof}

\begin{corollary}
\label{cor:general_critical(aa)}
Assume that $r$ is general.
Under Hypothesis~A and Notation~\ref{not:layers},
we may assume that the following hold.
\begin{enumerate}[\indent \rm (1)]
\item Every vertex $x$ in $M$ with $\deg_{M}(x)=4$
is either converging or diverging.
\item For every face $D_{j,i}$ of $M$,
one of the following hold.
\begin{enumerate}[\rm (a)]
\item Both $S(\phi(\partial D_{j,i}^+))$ and $S(\phi(\partial D_{j,i}^-))$
 contain $(m,S_1,m)$ as a subsequence.
\item Both $S(\phi(\partial D_{j,i}^+))$ and $S(\phi(\partial D_{j,i}^-))$
 contain $(m+1,S_2,m+1)$ as a subsequence.
\end{enumerate}
\item
For every $j$, both $CS(\phi(\alpha_j))$ and $CS(\phi(\delta_j^{-1}))$
consist of $m$ and $m+1$.
\end{enumerate}
\end{corollary}

\begin{proof}
(1) is nothing other than Proposition~\ref{prop:general_critical(aa)}(2).
(2) follows from Lemma~\ref{lem:two_cases:intermediate_layer},
Proposition~\ref{prop:general_critical(aa)}(1)
and the fact that $S_1$ (resp., $S_2$) begins and ends with $m+1$ (resp., $m$).
\end{proof}

\section{Key results for the induction}
\label{sec:result_for_induction}

In this section, we prove key results,
Propositions~\ref{prop:induction_general_1}
and \ref{prop:induction_general_3}
for $r=[m,m_2, \dots, m_k]$ with $m \ge 2$, $m_2 \ge 2$ and $k \ge 3$
and $r=[m,1,m_3, \dots, m_k]$ with $m \ge 2$ and $k \ge 4$, respectively,
used for the inductive proof of Main Theorem~\ref{main_theorem}
in Section~\ref{sec:proof_for_general_2-bridge_links}.
Throughout this section, we assume that Hypothesis~A holds
and that the annular diagram $M$ satisfies the conditions in
Corollary \ref{cor:general_critical(aa)}.

\subsection{The case for $r=[m,m_2, \dots, m_k]$ with $m \ge 2$, $m_2 \ge 2$ and $k \ge 3$}

We first establish a key result, Proposition~\ref{prop:induction_general_1},
for $r=[m,m_2, \dots, m_k]$, where $m \ge 2$, $m_2 \ge 2$ and $k \ge 3$.
Recall from Remark~\ref{rem:general_decomposition}(1) that
$CS(r)=\lp S_1, S_2, S_1, S_2 \rp$,
where $S_1$ begins and ends with $(m+1, (m_2-1) \langle m \rangle, m+1)$, and
$S_2$ begins and ends with $(m_2 \langle m \rangle)$.

\begin{lemma}
\label{lem:general_1(b)}
Let $r=[m,m_2, \dots, m_k]$, where $m \ge 2$, $m_2 \ge 2$ and $k \ge 3$.
Under Hypothesis~A and Notation~\ref{not:layers2},
the following hold for every $i$ and $j$.
\begin{enumerate}[\indent \rm (1)]
\item $S(z_{j,i,e}y_{j,i+1,b}) \neq (m+1, m+1)$.

\item $S(z_{j,i,e}'y_{j,i+1,b}') \neq (m+1, m+1)$.
\end{enumerate}
\end{lemma}

\begin{proof}
We prove only (1), because the proof of (2) is parallel.
Suppose on the contrary that $S(z_{j,1,e}y_{j,2,b})=(m+1, m+1)$
for some $j$.
First assume $j=0$.
If Hypothesis~B holds,
then $S(z_{0,1})$ begins with $(m_2 \langle m \rangle)$,
because $S_2$ begins and ends with $(m_2 \langle m \rangle)$
(see Remark~\ref{rem:general_decomposition}(1))
whereas $S(z_{0,1,e})=(m+1)$ by assumption.
This implies that $CS(\phi(\alpha_0))=CS(s)$ contains two consecutive $m$'s.
So $CS(s)$ contains two consecutive $m$'s and
two consecutive $m+1$'s, contradicting \cite[Lemma~3.8]{lee_sakuma_2}.
On the other hand, if Hypothesis~C holds,
then $CS(s)$ contains two consecutive $m$'s
(because it contains $S_2$)
and two consecutive $m+1$'s
by assumption, again contradicting \cite[Lemma~3.8]{lee_sakuma_2}.
Next assume $j \ge 1$.
By using Lemma~\ref{lem:vertex_position:intermediate_layer},
we can see that $S(\phi(e_{j,2}e_{j,3}))$ contains a subsequence $(m+1, m+1)$.
Thus $CS(\phi(\partial D_{j-1,2}))=CS(r)$ contains two consecutive $m+1$'s,
a contradiction.
\end{proof}

\begin{corollary}
\label{cor:general_1}
Let $r=[m,m_2, \dots, m_k]$, where $m \ge 2$, $m_2 \ge 2$ and $k \ge 3$.
Under Hypothesis~A and Notation~\ref{not:layers},
the following hold for every $j$.
\begin{enumerate}[\indent \rm (1)]
\item $CS(\phi(\alpha_j))$ does not contain $(m+1, m+1)$ as a subsequence.

\item $CS(\phi(\delta_j^{-1}))$ does not contain $(m+1, m+1)$ as a subsequence.
\end{enumerate}
\end{corollary}

\begin{proof}
We prove only (1), because the proof of (2) is parallel.
Suppose on the contrary that $CS(\phi(\alpha_j))$
contains a subsequence $(m+1, m+1)$ for some $j$.
Let $v$ be a subword of the cyclic word $(\phi(\alpha_j))$
corresponding to a subsequence $(m+1, m+1)$.
Note that $S(\phi(\partial D_{j,i}^+))$ does not contain $(m+1, m+1)$,
because $S(r)=S(\phi(\partial D_{j,i}))$ does not.
Thus the only possibility is that $v=z_{j,i,e}y_{j,i+1,b}$
for some $i$ by Corollary~\ref{cor:general_critical(aa)}.
But this is impossible by Lemma~\ref{lem:general_1(b)}(1).
\end{proof}

\begin{proposition}
\label{prop:induction_general_1}
Let $r=[m,m_2, \dots, m_k]$, where $m \ge 2$, $m_2 \ge 2$ and $k \ge 3$.
Suppose that there are two distinct rational numbers $s, s' \in I_1(r) \cup I_2(r)$
such that the unoriented loops $\alpha_{s}$ and $\alpha_{s'}$ are
homotopic in $S^3-K(r)$, namely suppose that Hypothesis~A holds.
Let $\tilde{r}=[m_2-1, m_3, \dots, m_k]$ be as in \cite[Lemma~3.11]{lee_sakuma_2}.
Then there are two distinct rational numbers
$\tilde{s}, \tilde{s}' \in I_1(\tilde{r}) \cup I_2(\tilde{r})$
such that the unoriented loops $\alpha_{\tilde{s}}$ and
$\alpha_{\tilde{s}'}$ are homotopic in $S^3-K(\tilde{r})$.
Moreover, there is a reduced
nontrivial
conjugacy diagram over $G(K(\tilde{r}))$ for
$\alpha_{\tilde{s}}$ and $\alpha_{\tilde{s}'}$
such that none of the degree $4$ vertices is mixing.
\end{proposition}

\begin{proof}
Recall from
Corollaries~\ref{cor:general_critical(aa)} and~\ref{cor:general_1}
that both $CS(\phi(\alpha_j))$ and $CS(\phi(\delta_j^{-1}))$ consist of $m$ and $m+1$
without a subsequence $(m+1, m+1)$ for every $j$.
In particular, both $CS(\phi(\alpha_0))=CS(s)$ and
$CS(\phi(\delta_p^{-1}))=CS(s')$ consist of $m$ and $m+1$
without a subsequence $(m+1, m+1)$.
This implies that if $s=[p_1, p_2, \dots, p_h]$ and $s'=[q_1, q_2, \dots, q_l]$,
where $p_i, q_j \in \ZZ_+$ and $p_h, q_l \ge 2$, then
$p_1=q_1=m$ and $p_2, q_2 \ge 2$.
Put $\tilde{s}=[p_2-1, p_3, \dots, p_h]$ and $\tilde{s}'=[q_2-1, q_3, \dots, q_l]$
as in \cite[Lemma~3.11]{lee_sakuma_2}.

\medskip
{\bf Claim.}
{\it Both $\tilde{s}$ and $\tilde{s}'$ belong to
$I_1(\tilde{r}) \cup I_2(\tilde{r})$.}

\begin{proof}{\it of Claim }
Since $p_1=q_1=m$, we have
\[
\tilde{s}=
\cfrac{1}{
\cfrac{1}{-p_1+
\cfrac{1}{s}}-1}
=
\cfrac{1}{
\cfrac{1}{-m+
\cfrac{1}{s}}-1}
\quad
\textrm{and}
\quad
\tilde{s}'=
\cfrac{1}{
\cfrac{1}{-q_1+
\cfrac{1}{s'}}-1}
=
\cfrac{1}{
\cfrac{1}{-m+
\cfrac{1}{s'}}-1}.
\]
Put $I_1(r)=[0,r_1]$ and $I_2(r)=[r_2,1]$.
Recall from \cite[Section~2]{lee_sakuma_2} that
\begin{align*}
r_1 &=
\begin{cases}
[m, m_2, \dots, m_{k-1}] & \mbox{if $k$ is odd,}\\
[m, m_2, \dots, m_{k-1}, m_k-1] & \mbox{if $k$ is even,}
\end{cases}\\
r_2 &=
\begin{cases}
[m, m_2, \dots, m_{k-1}, m_k-1] & \mbox{if $k$ is odd,}\\
[m, m_2, \dots, m_{k-1}] & \mbox{if $k$ is even.}
\end{cases}
\end{align*}
Also put $I_1(\tilde{r})=[0, t_1]$ and $I_2(\tilde{r})=[t_2,1]$;
then we have
\begin{align*}
t_1 &=
\begin{cases}
[m_2-1, \dots, m_{k-1}, m_k-1] & \mbox{if $k$ is odd,}\\
[m_2-1, \dots, m_{k-1}] & \mbox{if $k$ is even,}
\end{cases}
\\
t_2 &=
\begin{cases}
[m_2-1, \dots, m_{k-1}] & \mbox{if $k$ is odd,}\\
[m_2-1, \dots, m_{k-1}, m_k-1] & \mbox{if $k$ is even.}
\end{cases}
\end{align*}
It then follows that
\[
t_1=
\cfrac{1}{
\cfrac{1}{-m+
\cfrac{1}{r_2}}-1}
\quad
\textrm{and}
\quad
t_2=
\cfrac{1}{
\cfrac{1}{-m+
\cfrac{1}{r_1}}-1}.
\]
Therefore the fact $s, s' \in I_1(r) \cup I_2(r)$ yields
$\tilde{s}, \tilde{s}' \in I_1(\tilde{r}) \cup I_2(\tilde{r})$.
\end{proof}

Let $\tilde{R}$ be the symmetrized subset of $F(a, b)$ generated by the single relator
$u_{\tilde{r}}$ of the upper presentation $G(K(\tilde{r}))=\langle a, b \, | \, u_{\tilde{r}} \rangle$.
Then as described below,
we can construct a reduced
annular $\tilde{R}$-diagram $\tilde{M}$
such that $u_{\tilde{s}}$ is an outer boundary label and
$u_{\tilde{s}'}^{\pm 1}$ is an inner boundary label of $\tilde{M}$.
This proves that the unoriented loops
$\alpha_{\tilde{s}}$ and $\alpha_{\tilde{s}'}$
are homotopic in $S^3-K(\tilde{r})$.
\end{proof}

We shall describe the explicit construction of
a reduced annular $\tilde{R}$-diagram $\tilde{M}$
from $M$. To this end, we introduce the following definition.

\begin{definition}
\label{def:T-sequence_1}
\rm
Suppose $r=[m,m_2, \dots, m_k]$, where $m \ge 2$, $m_2 \ge 2$ and $k \ge 3$.
Let $w$ be an alternating word in $\{a,b\}$, and suppose that
$S(w)=(a_1,a_2,\cdots,a_k)$ is a finite sequence
consisting of $m$ and $m+1$, which does not contain $(m+1,m+1)$.
Then we define the $T$-sequence of $w$, denoted by $T(w)$, and the cyclic $T$-sequence,
denoted by $CT(w)$, as follows.
Express $S(w)$ as
\[
(*, t_1\langle m\rangle, m+1, t_2\langle m\rangle,
\dots,m+1, t_s\langle m\rangle,*'),
\]
where each of $*$ and $*'$ is either $m+1$ or $\emptyset$
and $(t_1,t_2,\dots,t_s)$ is a sequence of positive integers.
Then $T(w)$ is defined to be the sequence $(t_1,\cdots, t_s)$.
If precisely one of $*$ and $*'$ is $m+1$ and the other is $\emptyset$,
we define $CT(w)$ to be the cyclic sequence $\lp t_1,\cdots, t_s\rp$.
If this $w$ represents a reduced cyclic word $u=(w)$, then
we define the cyclic sequence $CT(u)$ by $CT(w)$.
\end{definition}

Under Hypothesis~A and Notation~\ref{not:layers},
by Corollaries~\ref{cor:general_critical(aa)} and~\ref{cor:general_1},
both $CS(\phi(\alpha_j))$ and $CS(\phi(\delta_j^{-1}))$
consist of $m$ and $m+1$ without a subsequence $(m+1, m+1)$,
so the cyclic sequences $CT(\phi(\alpha_j))$ and $CT(\phi(\delta_j^{-1}))$ are well-defined
for every $j$.
Recall that every vertex in $M$ of degree $4$ is
assumed to be converging or diverging by Corollary~\ref{cor:general_critical(aa)}.
Moreover, we can also assume,
by using Corollary~\ref{cor:general_critical(aa)}(2),
that every degree $2$ vertex of $M$ is
also either converging or diverging.
(In the new diagram, it may happen that some $\phi(e_{0,i})$ or some $\phi(e_{p,i}')$
is not a piece and so \cite[Convention~4.6]{lee_sakuma_2} is not satisfied.
But this does not affect the arguments in this section.)
Hence $S(\phi(e_{j,i}))$ is a subsequence of $S(\phi(\alpha_j))$,
which consists of $m$ and $m+1$ and does not contain $(m+1,m+1)$
as a subsequence.
Thus the $T$-sequence of $\phi(e_{j,i})$ is also well-defined for every $i$ and $j$.
Similarly the $T$-sequence of $\phi(e_{j,i}')$ is also well-defined for every $i$ and $j$.
By Corollary~\ref{cor:general_critical(aa)}(2),
we may assume that $T(\phi(e_{j,i}))$ and $T(\phi(e_{j,i}'))$
are nonempty for every $i$ and $j$.

Now we construct a reduced annular $\tilde{R}$-diagram $(\tilde M,\psi)$
from the annular $R$-diagram $(M,\phi)$
by taking $T$-sequences of the boundary labels, as follows.
Take the underlying map of $\tilde{M}$ being the same as that of $M$.
For every $i$ and $j$, by $\tilde{e}_{j,i}$ denote the edge of $\tilde{M}$
which corresponds to the edge $e_{j,i}$ of $M$, and assign an alternating word,
$\psi(\tilde{e}_{j,i})$, in $\{a,b\}$ to $\tilde{e}_{j,i}$ in the following order.

\medskip
\noindent
{\bf Step 1.}
For each $i=1, \dots, 2t$, assign $\psi(\tilde{e}_{0,i})$ so that
$\psi(\tilde{e}_{0,1} \cdots \tilde{e}_{0,i}):=
\psi(\tilde{e}_{0,1}) \cdots \psi(\tilde{e}_{0,i})$ is an alternating word such that
\[
S(\psi(\tilde{e}_{0,1} \cdots \tilde{e}_{0,i}))=T(\phi(e_{0,1} \cdots e_{0,i})).
\]
Once this assignment is done, we see the following.
\begin{enumerate}[\indent \rm (i)]
\item The word $\psi(\tilde{e}_{0,1} \cdots \tilde{e}_{0,2t})$ is cyclically alternating,
because the sum of the terms of $CT(s)=CS(\tilde{s})$ is even.

\item $CS(\psi(\tilde{e}_{0,1} \cdots \tilde{e}_{0,2t}))=
CT(\phi(e_{0,1} \cdots e_{0,2t}))=CT(\phi(\alpha_0))=CT(s)=CS(\tilde{s})$,
because $CS(\tilde{s})$ has even number of terms.
\end{enumerate}

\medskip
\noindent
{\bf Step 2.}
Assign $\psi(\tilde{e}_{0,1}')$ so that
$\psi(\tilde{e}_{0,1}^{-1} \tilde{e}_{0,1}')$ is an alternating word such that
$S(\psi(\tilde{e}_{0,1}^{-1} \tilde{e}_{0,1}'))=T(\phi(e_{0,1}^{-1} e_{0,1}'))$, and
assign $\psi(\tilde{e}_{0,2}')$ so that
$\psi(\tilde{e}_{0,1}' \tilde{e}_{0,2}')$ is an alternating word such that
$S(\psi(\tilde{e}_{0,1}' \tilde{e}_{0,2}'))=T(\phi(e_{0,1}' e_{0,2}'))$.
Once this assignment is done, we see the following.
\begin{enumerate}[\indent \rm (i)]
\item The word $\psi(\tilde{e}_{0,1} \tilde{e}_{0,2} \tilde{e}_{0,2}'^{-1} \tilde{e}_{0,1}'^{-1})$
is cyclically alternating,
because the sum of the terms of $CT(r)=CS(\tilde r)$ is even.

\item $CS(\psi(\tilde{e}_{0,1} \tilde{e}_{0,2} \tilde{e}_{0,2}'^{-1} \tilde{e}_{0,1}'^{-1}))
=CT(\psi(e_{0,1} e_{0,2}e_{0,2}'^{-1} e_{0,1}'^{-1}))=CT(r)=CS(\tilde{r})$,
because $CS(\tilde{r})$ has even number of terms.
\end{enumerate}

\medskip
\noindent
{\bf Step 3.}
For $i=2, \dots, t$, assign $\psi(\tilde{e}_{0,2i-1}')$
so that $\psi(\tilde{e}_{0,2i-1}^{-1} \tilde{e}_{0,2i-1}')$ is an alternating word such that
$S(\psi(\tilde{e}_{0,2i-1}^{-1} \tilde{e}_{0,2i-1}'))=T(\phi(e_{0,2i-1}^{-1} e_{0,2i-1}'))$,
and assign $\psi(\tilde{e}_{0,2i}')$ so that
$\psi(\tilde{e}_{0,2i-1}' \tilde{e}_{0,2i}')$ is an alternating word such that
$S(\psi(\tilde{e}_{0,2i-1}' \tilde{e}_{0,2i}'))=T(\phi(e_{0,2i-1}' e_{0,2i}'))$.
Then we see the following.
\begin{enumerate}[\indent \rm (i)]
\item $S(\psi(\tilde{e}_{0,1}' \cdots \tilde{e}_{0,2i-1}'))=T(\phi(e_{0,1}' \cdots e_{0,2i-1}'))$,
because of the reason described after this list.

\item The word $\psi(\tilde{e}_{0,2i-1} \tilde{e}_{0,2i} \tilde{e}_{0,2i}'^{-1} \tilde{e}_{0,2i-1}'^{-1})$
is cyclically alternating,
because the sum of the terms of $CT(r)=CS(\tilde r)$ is even.

\item $CS(\psi(\tilde{e}_{0,2i-1} \tilde{e}_{0,2i} \tilde{e}_{0,2i}'^{-1} \tilde{e}_{0,2i-1}'^{-1}))
=CT(\phi(e_{0,2i-1} e_{0,2i} e_{0,2i}'^{-1} e_{0,2i-1}'^{-1}))=CT(r)=CS(\tilde{r})$,
because $CS(\tilde{r})$ has even number of terms.

\item The word $\psi(\tilde{e}_{0,1}' \cdots \tilde{e}_{0,2t}')$ is cyclically alternating,
because the sum of the terms of $CT(\phi(\delta_0))$ is even.

\item $CS(\psi(\tilde{e}_{0,1}' \cdots \tilde{e}_{0,2t}'))=
CT(\phi(e_{0,1}' \cdots e_{0,2t}'))=CT(\phi(\delta_0))$,
because $CT(\phi(\delta_0))$ has even number of terms.
\end{enumerate}
In the above, condition (i) is verified as follows.
We explain the reason when $i=2$. (The other cases can be treated similarly.)
Since $CS(\phi(\alpha_0))=CS(s)$ and $CS(\phi(\partial D_1))=CS(\phi(\partial D_2))=CS(r)$
consist of $m$ and $m+1$ and do not contain $(m+1,m+1)$,
we have the four possibilities around the vertex between $D_1$ and $D_2$
as described in the left figures in Figure~\ref{fig.taking_tsequence_1},
up to reflection in the vertical line passing through the vertex.
In each of the right figure, we may assume without loss of generality
that the upper left segment is oriented
so that it is converging into the vertex.
Then the orientations of the three remaining segments
in each of the right figures are specified
by the requirements in Steps~1 and 2 and the the new requirement
$S(\psi(\tilde{e}_{0,3}^{-1} \tilde{e}_{0,3}'))=T(\phi(e_{0,3}^{-1} e_{0,3}'))$.
In each case, we can check that the condition
$S(\psi(\tilde{e}_{0,1}' \tilde{e}_{0,2}' \tilde{e}_{0,3}'))=
T(\phi(e_{0,1}'e_{0,2}' e_{0,3}'))$ holds.

\begin{figure}[h]
\includegraphics{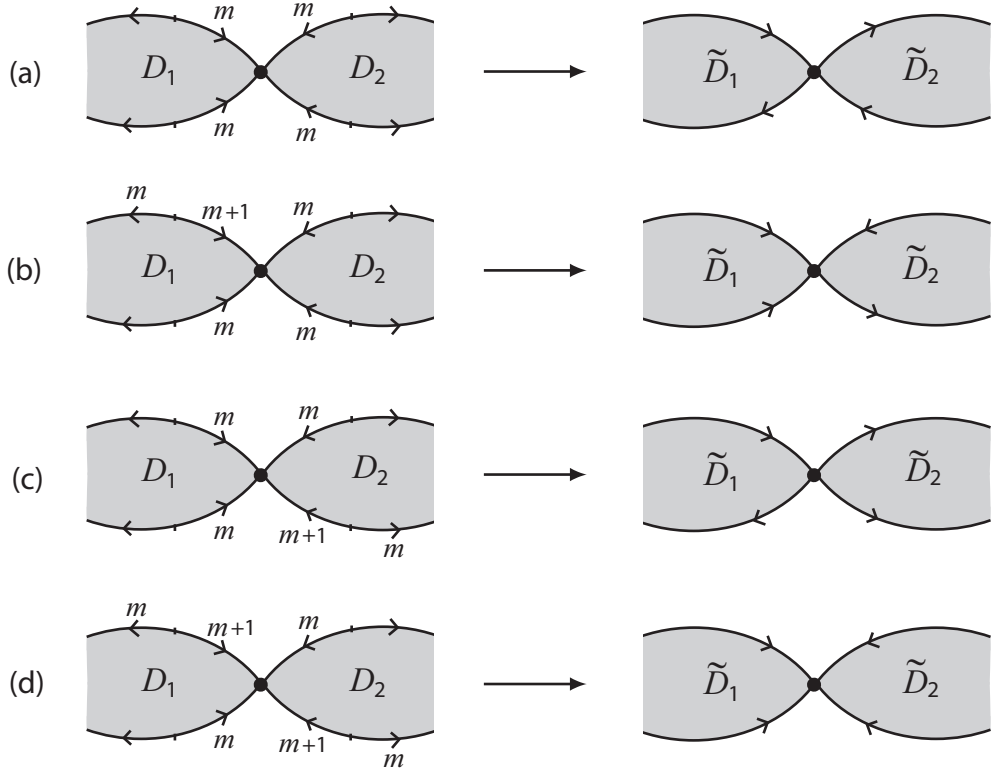}
\caption{
Step~3 of the construction of $\tilde{M}$ from $M$}
\label{fig.taking_tsequence_1}
\end{figure}

\medskip
\noindent
{\bf Step 4.}
For each $j=1, \dots, p$, repeat Steps~2 and 3 to obtain the following.
\begin{enumerate}[\indent \rm (i)]
\item The word $\psi(\tilde{e}_{j,2i-1} \tilde{e}_{j,2i} \tilde{e}_{j,2i}'^{-1} \tilde{e}_{j,2i-1}'^{-1})$
is cyclically alternating for every $i=1, \dots, t$.

\item $CS(\psi(\tilde{e}_{j,2i-1} \tilde{e}_{j,2i} \tilde{e}_{j,2i}'^{-1} \tilde{e}_{j,2i-1}'^{-1}))
=CT(\phi(e_{j,2i-1} e_{j,2i}e_{j,2i}'^{-1}e_{j,2i-1}'^{-1}))=CT(r)=CS(\tilde{r})$ for every $i=1, \dots, t$.

\item The word $\psi(\tilde{e}_{p,1}' \cdots \tilde{e}_{p,2t}')$ is cyclically alternating.

\item $CS(\psi(\tilde{e}_{p,1}' \cdots \tilde{e}_{p,2t}'))=CT(\phi(e_{p,1}' \cdots e_{p,2t}'))
=CT(\psi(\delta_p))=CT(s')=CS(\tilde{s}')$.
\end{enumerate}

It is obvious from the construction (see Figure~\ref{fig.taking_tsequence_1})
that none of the degree $4$ vertices of
the diagram $\tilde M$ constructed above is mixing.

Finally we show that $\tilde M$ is reduced.
Suppose on the contrary that $\tilde M$ is not reduced.
Then there is
a pair of faces, say ${\tilde D}$ and ${\tilde D}'$, in $\tilde M$
having a common edge, say ${\tilde e}=\partial {\tilde D} \cap \partial {\tilde D}'$,
such that $\psi(\delta_1) \equiv \psi(\delta_2)^{-1}$, where
${\tilde e} \delta_1$ and $\delta_2 {\tilde e}^{-1}$
are boundary cycles of ${\tilde D}$ and ${\tilde D}'$, respectively.
Then we see that the corresponding faces $D$ and $D'$ in $M$ have a common edge
$e=\partial D \cap \partial D'$ such that $S(\phi(e\gamma_1))=S(\phi(e\gamma_2^{-1}))$, where
$e \gamma_1$ and $\gamma_2 e^{-1}$ are boundary cycles of
$D$ and $D'$, respectively. So two words $\phi(e\gamma_1) \equiv \phi(e)\phi(\gamma_1)$
and $\phi(e\gamma_2^{-1}) \equiv \phi(e)\phi(\gamma_2)^{-1}$ have the same initial letter and
the same associated $S$-sequence. Then by \cite[Lemma~3.5(1)]{lee_sakuma_2}
$\phi(e)\phi(\gamma_1) \equiv \phi(e)\phi(\gamma_2)^{-1}$, so
$\phi(\gamma_1) \equiv \phi(\gamma_2)^{-1}$, which contradicts the fact that $M$ is reduced.

\subsection{The case for $r=[m,1,m_3, \dots, m_k]$ with $m \ge 2$ and $k \ge 4$}

We next establish a key result, Proposition~\ref{prop:induction_general_3},
for $r=[m,1,m_3, \dots, m_k]$, where $m \ge 2$ and $k \ge 4$.
Recall from Remark~\ref{rem:general_decomposition}(2) that
$CS(r)=\lp S_1, S_2, S_1, S_2 \rp$,
where $S_1$ begins and ends with $((m_3+1) \langle m+1 \rangle)$,
and $S_2$ begins and ends with $(m, m_3 \langle m+1 \rangle, m)$.

\begin{lemma}
\label{lem:general_3(b)}
Let $r=[m,1,m_3, \dots, m_k]$, where $m \ge 2$ and $k \ge 4$.
Under Hypothesis~A and Notation~\ref{not:layers2},
the following hold for every $i$ and $j$.
\begin{enumerate}[\indent \rm (1)]
\item $S(z_{j,i,e}y_{j,i+1,b}) \neq (m, m)$.

\item $S(z_{j,i,e}'y_{j,i+1,b}') \neq (m, m)$.
\end{enumerate}
\end{lemma}

\begin{proof}
We prove only (1), because the proof of (2) is parallel.
Suppose on the contrary that $S(z_{j,1,e}y_{j,2,b})=(m, m)$
for some $j$.
First assume $j=0$.
If Hypothesis~B holds, then
$CS(\phi(\alpha_0))=CS(s)$ contains
two consecutive $m+1$'s (because it contains $S_1$)
and two consecutive $m$'s by assumption.
This contradicts \cite[Lemma~3.8]{lee_sakuma_2}.
On the other hand, if Hypothesis~C holds,
then $S(z_{0,1})$ begins with $((m_3+1) \langle m+1 \rangle)$,
since $S_1$ begins and ends with $((m_3+1) \langle m+1 \rangle)$
whereas $S(z_{0,1,e})=m$ by assumption.
This implies that $CS(s)$
contains two consecutive $m$'s and two consecutive $m+1$'s,
again contradicting \cite[Lemma~3.8]{lee_sakuma_2}.
Next assume $j \ge 1$.
By using Lemma~\ref{lem:vertex_position:intermediate_layer},
we can see that $S(\phi(e_{j,2}e_{j,3}))$ contains a subsequence of the form
$(\ell_1, m, m, \ell_2)$ with $\ell_1, \ell_2 \in \ZZ_+$.
Hence $CS(\phi(\partial D_{j-1,2}))=CS(r)$ contains two consecutive $m$'s,
a contradiction.
\end{proof}

\begin{corollary}
\label{cor:general_3}
Let $r=[m,1,m_3, \dots, m_k]$, where $m \ge 2$ and $k \ge 4$.
Under Hypothesis~A and Notation~\ref{not:layers},
the following hold for every $j$.
\begin{enumerate}[\indent \rm (1)]
\item No two consecutive terms of $CS(\phi(\alpha_j))$ can be $(m, m)$.

\item No two consecutive terms of $CS(\phi(\delta_j^{-1}))$ can be $(m, m)$
for every $j=0, \dots, p$.
\end{enumerate}
\end{corollary}

\begin{proof}
We prove only (1), because the proof of (2) is parallel.
Suppose on the contrary that $CS(\phi(\alpha_j))$
contains a subsequence $(m, m)$ for some $j$.
Let $v$ be a subword of the cyclic word $(\phi(\alpha_j))$
corresponding to a subsequence $(m, m)$.
Note that $S(\phi(\partial D_{j,i}^+))$ does not contain $(m, m)$,
because $S(r)=S(\phi(\partial D_{j,i}))$ does not.
Thus the only possibility is that $v=z_{j,i,e}y_{j,i+1,b}$
for some $i$ by
Corollary~\ref{cor:general_critical(aa)}.
But this is impossible by Lemma~\ref{lem:general_3(b)}(1).
\end{proof}

\begin{definition}
\label{def:T-sequence_2}
\rm
Suppose $r=[m,1,m_3, \dots, m_k]$, where $m \ge 2$ and $k \ge 4$.
Let $w$ be an alternating word in $\{a,b\}$, and suppose that
$S(w)=(a_1,a_2,\cdots,a_k)$ is a finite sequence
consisting of $m$ and $m+1$, which does not contain $(m,m)$.
Then we define the $T$-sequence of $w$, denoted by $T(w)$, and the cyclic $T$-sequence,
denoted by $CT(w)$, as follows.
Express $S(w)$ as
\[
(*, t_1\langle m+1\rangle, m, t_2\langle m+1\rangle,
\dots,m, t_s\langle m+1\rangle,*'),
\]
where each of $*$ and $*'$ is either $m$ or $\emptyset$
and $(t_1,t_2,\dots,t_s)$ is a sequence of positive integers.
Then $T(w)$ is defined to be the sequence $(t_1,\cdots, t_s)$.
If precisely one of $*$ and $*'$ is $m$ and the other is $\emptyset$,
we define $CT(w)$ to be the cyclic sequence $\lp t_1,\cdots, t_s\rp$.
If this $w$ represents a reduced cyclic word $u=(w)$, then
we define the cyclic sequence $CT(u)$ by $CT(w)$.
\end{definition}

Under Hypothesis~A and Notation~\ref{not:layers},
by Corollaries~\ref{cor:general_critical(aa)} and~\ref{cor:general_3},
both $CS(\phi(\alpha_j))$ and $CS(\phi(\delta_j^{-1}))$
consist of $m$ and $m+1$ without a subsequence $(m, m)$,
so the cyclic sequences $CT(\phi(\alpha_j))$ and $CT(\phi(\delta_j^{-1}))$ are well-defined
for every $j$.
Clearly we may assume that every degree $2$ vertex of $M$ is either
converging or diverging. Moreover since every vertex in $M$ of degree $4$ is
assumed to be converging or diverging by Corollary~\ref{cor:general_critical(aa)},
the $T$-sequence of $\phi(e_{j,i})$ is also well-defined for every $i$ and $j$.
Then as in the previous case, we can construct an annular $\tilde{R}$-diagram $\tilde M$
from the annular $R$-diagram $M$
by taking $T$-sequences of the boundary labels.

\begin{proposition}
\label{prop:induction_general_3}
Let $r=[m,1,m_3, \dots, m_k]$, where $m \ge 2$ and $k \ge 4$.
Suppose that there are two distinct rational numbers $s, s' \in I_1(r) \cup I_2(r)$
such that the unoriented loops $\alpha_{s}$ and $\alpha_{s'}$ are
homotopic in $S^3-K(r)$, namely suppose that Hypothesis~A holds.
Let $\tilde{r}=[m_3, \dots, m_k]$ be as in \cite[Lemma~3.11]{lee_sakuma_2}.
Then there are two distinct rational numbers
$\tilde{s}, \tilde{s}' \in I_1(\tilde{r}) \cup I_2(\tilde{r})$
such that the unoriented loops $\alpha_{\tilde{s}}$ and
$\alpha_{\tilde{s}'}$ are homotopic in $S^3-K(\tilde{r})$.
Moreover, there is a reduced
conjugacy diagram over $G(K(\tilde{r}))$ for
$\alpha_{\tilde{s}}$ and $\alpha_{\tilde{s}'}$
such that none of the degree $4$ vertices is mixing.
\end{proposition}

\begin{proof}
Recall from
Corollaries~\ref{cor:general_critical(aa)} and~\ref{cor:general_3} that
both $CS(\phi(\alpha_j))$ and $CS(\phi(\delta_j^{-1}))$ consist of $m$ and $m+1$
without a subsequence $(m, m)$, for every $j$.
In particular, both $CS(\phi(\alpha_0))=CS(s)$ and
$CS(\phi(\delta_p^{-1}))=CS(s')$ consist of $m$ and $m+1$
without a subsequence $(m, m)$.
This implies that if $s=[p_1, p_2, \dots, p_h]$ and $s'=[q_1, q_2, \dots, q_l]$,
where $p_i, q_j \in \ZZ_+$ and $p_h, q_l \ge 2$, then
$p_1=q_1=m$, $p_2=q_2=1$ and $h, l \ge 3$.
Put $\tilde{s}=[p_3, \dots, p_h]$ and $\tilde{s}'=[q_3, \dots, q_l]$
as in \cite[Lemma~3.11]{lee_sakuma_2}.
Let $\tilde{R}$ be the symmetrized subset of $F(a, b)$ generated by the single relator
$u_{\tilde{r}}$ of the upper presentation $G(K(\tilde{r}))=\langle a, b \, | \, u_{\tilde{r}} \rangle$.
Then, as in the previous case,
we can construct a reduced annular $\tilde{R}$-diagram
$(\tilde{M},\psi)$
such that $u_{\tilde{s}}$
is an outer boundary label and $u_{\tilde{s}'}^{\pm 1}$ is an inner boundary label
of $\tilde{M}$.
This proves that the unoriented loops $\alpha_{\tilde{s}}$ and $\alpha_{\tilde{s}'}$
are homotopic in $S^3-K(\tilde{r})$.
Moreover, we can also see that
$\tilde M$ is reduced and that
none of the degree $4$ vertices of
the diagram $\tilde M$ constructed above
is mixing.
\end{proof}

\section{Proof of Main Theorem~\ref{main_theorem} for the general cases}
\label{sec:proof_for_general_2-bridge_links}

In this section, we prove Main Theorem~\ref{main_theorem}
when $r$ is general.
To this end, we use the following terminology.
\begin{enumerate}[\indent \rm (1)]
\item A rational number $r$ with $0< r\le 1/2$ is {\it special},
if it is equal to $1/p=[p]$ with $p\ge 2$, $[m,n]$, or $[m,1,n]$
with $m,n\ge 2$.

\item A rational number $r$ with $1/2< r<1$ is {\it special},
if $1-r$ is special.

\item A rational number $r$ with $0< r<1$ is {\it general},
if it is not special.
\end{enumerate}

The following proposition
forms the starting point of the inductive proof of
Main Theorem~\ref{main_theorem}.

\begin{proposition}
\label{prop:induction_base}
Let $r$ be a special rational number with $0<r<1$.
Then the following is the complete list of
pairs of mutually distinct elements $\{s,s'\}$ of $I_1(r)\cup I_2(r)$
such that $\alpha_{s}$ and $\alpha_{s'}$ are homotopic in $S^3-K(r)$.
\begin{enumerate}[\indent \rm(1)]
\item $r=1/p$ and the set $\{s, s'\}$ equals
$\{q_1/p_1, q_2/p_2\}$,
where $p \ge 2$ is an integer, and
$s=q_1/p_1$ and $s'=q_2/p_2$ satisfy
$q_1=q_2$ and $q_1/(p_1+p_2)=1/p$, where $(p_i, q_i)$ is a pair of
relatively prime positive integers.

\item $r=3/8=[2, 1, 2]$ and the set $\{s, s'\}$ equals
either $\{1/6, 3/10\}$ or $\{3/4, 5/12\}$.

\item $r=1-1/p=[1,p-1]$ and the set $\{s, s'\}$ equals
$\{1-q_1/p_1, 1-q_2/p_2\}$,
where $p \ge 2$ is an integer, and
$q_1=q_2$ and $q_1/(p_1+p_2)=1/p$, where $(p_i, q_i)$ is a pair of
relatively prime positive integers.

\item $r=1-3/8=[1,1,1,2]$ and the set $\{s, s'\}$ equals
either $\{1-1/6, 1-3/10\}$ or $\{1-3/4, 1-5/12\}$.
\end{enumerate}
Moreover, any reduced conjugacy diagram over $G(K(r))$
for $\alpha_{s}$ and $\alpha_{s'}$ contains a vertex which is mixing.
\end{proposition}

\begin{proof}
Suppose $0<r \le 1/2$.
Then by \cite[Main Theorem~2.7]{lee_sakuma_2},
\cite[Main Theorems~2.2 and~2.3]{lee_sakuma_3}
and the results in Sections~\ref{sec:proof_of_main_theorem_1} and \ref{sec:proof_of_main_theorem_3},
we see that (1) or (2) holds.
We prove the assertion for the conjugacy diagram in this case.
Suppose that (1) holds.
Then the outer boundary layer of the conjugacy diagram
should be as depicted
in Figure~\ref{fig.general_proof_0}
by \cite[Section~5]{lee_sakuma_2}.
Since all vertices of the diagram are mixing, we obtain the desired result.
Suppose that (2) holds.
Then, by \cite[Section~6, in particular Figures~21(a) and 23(a)]{lee_sakuma_3},
the conjugacy diagram should be as depicted in Figure~\ref{fig.general_proof_1}.
Again, every vertex of degree $4$ is mixing, and so
we obtain the desired result.

\begin{figure}[h]
\includegraphics{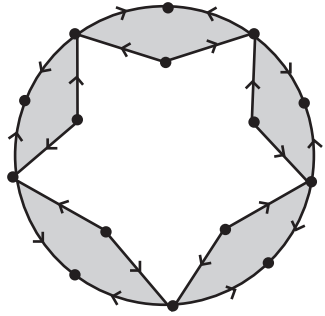}
\caption{
The outer boundary layer of any of the conjugacy diagrams for the case $r=1/p$}
\label{fig.general_proof_0}
\end{figure}

\begin{figure}[h]
\includegraphics{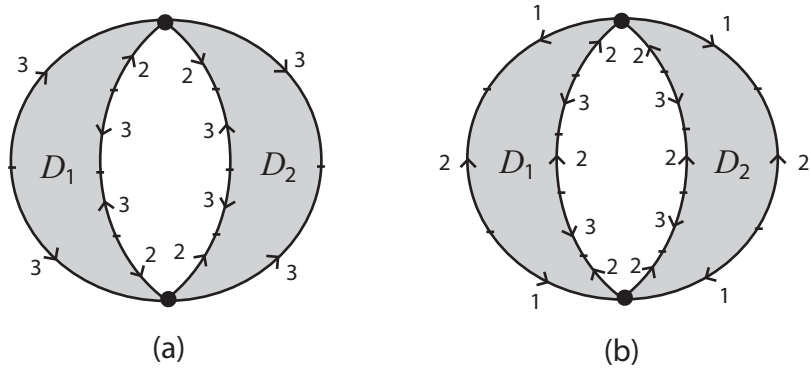}
\caption{
The conjugacy diagrams for the case $r=[2,1,2]$,
where $\{s,s'\}$ is (a) $\{1/6,3/10\}$ and
(b) $\{3/4,5/12\}$}
\label{fig.general_proof_1}
\end{figure}

Suppose $1/2< r<1$.
Note that there is a homeomorphism $f:(S^3,K(r))\to (S^3,K(1-r))$
preserving the bridge sphere such that
$f(\alpha_s)=\alpha_{1-s}$ and
that $f$ induces an isomorphism from
$G(K(r))$ to $G(K(1-r))$
sending the standard generators $a$ and $b$ to $a$ and $b^{-1}$,
respectively.
In fact, such a homeomorphism is obtained as the composition
of the natural homeomorphisms
\[
(S^3,K(r)) \to (S^3,K(-r))\to (S^3,K(1-r)),
\]
where the latter homeomorphism is explained in
\cite[the end of Section~3]{lee_sakuma}.
Thus,
by \cite[Main Theorem~2.7]{lee_sakuma_2},
\cite[Main Theorems~2.2 and~2.3]{lee_sakuma_3}
and the results in Sections
~\ref{sec:proof_of_main_theorem_1} and \ref{sec:proof_of_main_theorem_3},
we see that (3) or (4) holds.
Moreover, the conjugacy diagram over $G(K(r))$ is obtained as
the isomorphic image of the conjugacy diagram over $G(K(1-r))$.
Since the image of a mixing vertex by the isomorphism is again
a mixing vertex, we obtain the last assertion of the proposition.
\end{proof}

\begin{proof}{\it of Main Theorem~\ref{main_theorem} for the case when
$r$ is general }
Let $r$ be a general rational number with $0<r\le 1/2$,
namely either $r=[m,m_2, \dots, m_k]$, where $m \ge 2$, $m_2 \ge 2$ and $k \ge 3$,
or $r=[m, 1, m_3, \dots, m_k]$, where $m \ge 2$ and $k \ge 4$.
Suppose on the contrary that there exist two distinct
rational numbers $s, s' \in I_1(r) \cup I_2(r)$
such that $\alpha_s$ and $\alpha_{s'}$ are homotopic in $S^3-K(r)$,
namely suppose that Hypothesis~A is satisfied.
Let $\tilde{r}$ be as in \cite[Lemma~3.11]{lee_sakuma_2}.
Then by Propositions~\ref{prop:induction_general_1}
and~\ref{prop:induction_general_3},
there are two distinct rational numbers
$\tilde{s}, \tilde{s}' \in I_1(\tilde{r}) \cup I_2(\tilde{r})$
such that the unoriented loops $\alpha_{\tilde{s}}$ and
$\alpha_{\tilde{s}'}$ are homotopic in $S^3-K(\tilde{r})$.
Moreover, there is a reduced conjugacy diagram over $G(K(\tilde{r}))$ for
$\alpha_{\tilde{s}}$ and $\alpha_{\tilde{s}'}$,
such that none of the degree $4$ vertices is mixing.
Regardless of the type of $r$,
we put $\tilde{r}=[n_1, \dots, n_t]$, where $t \ge 2$,
each $n_i \in \ZZ_+$ and $n_t\ge 2$.
We proceed the proof by induction on $t \ge 2$.

\medskip
\noindent
{\bf Case 1.} {\it $t=2$, i.e., $\tilde{r}=[n_1,n_2]$.}
Then $\tilde r$ is special,
and hence Proposition~\ref{prop:induction_base} shows that
any reduced conjugacy diagram over $G(K(\tilde r))$ for
$\alpha_{\tilde{s}}$ and $\alpha_{\tilde{s}'}$ contains a
mixing vertex.
This contradicts to the fact observed in the above that there is a
reduced conjugacy diagram over $G(K(\tilde r))$ for
$\alpha_{\tilde{s}}$ and $\alpha_{\tilde{s}'}$
which has no mixing vertex.

\medskip
\noindent
{\bf Case 2.} {\it $t=3$, i.e., $\tilde{r}=[n_1,n_2,n_3]$.}
If $n_1 \ge 2$ and $n_2 \ge 2$, then $\tilde{r}$ is general
and the rational number $\tilde{\tilde{r}}=[n_2-1, n_3]$ is
as in Case 1.
So, this is impossible by the conclusion in Case 1.
If $n_1 \ge 2$ and $n_2=1$,
then $\tilde r$ is special and so
Proposition~\ref{prop:induction_base} implies that
any reduced conjugacy diagram over $G(K(\tilde r))$ for
$\alpha_{\tilde{s}}$ and $\alpha_{\tilde{s}'}$ contains a
mixing vertex, a contradiction.
If $n_1=1$, then $\tilde{r}$ is special, because
$1-\tilde{r}=[n_2+1,n_3]$ is special;
so we obtain a similar contradiction by Proposition~\ref{prop:induction_base}.
(To be precise, this rational number does not belong to the list in
Proposition~\ref{prop:induction_base}, which is also a contradiction.)

\medskip
\noindent
{\bf Case 3.} {\it $t=4$, i.e., $\tilde{r}=[n_1, n_2, n_3, n_4]$.}
If $n_1 \ge 2$, then $\tilde{r}$ is general and
the rational number $\tilde{\tilde{r}}$ is as in Case~1 or Case~2
according to whether $n_2=1$ or $n_2 \ge 2$.
So, this is impossible by the conclusions in Cases~1 and 2.
Hence we have $n_1=1$ and so let $\tilde{r}':=1-\tilde{r}=[n_2+1,n_3, n_4]$.
If $n_3 \ge 2$, then $\tilde{r}'$ is general
and the rational number $\tilde{\tilde{r}}'=[n_3-1,n_4]$ is as in Case~1.
So, this is impossible by the conclusion in Case~1.
If $n_3=1$, then $\tilde{r}$ is special, because
$1-\tilde{r}=\tilde{r}'$ is special;
so we obtain a contradiction by Proposition~\ref{prop:induction_base}.
(In this case, $\tilde{r}$ is as in
Proposition~\ref{prop:induction_base}(4).)

\medskip
\noindent
{\bf Case 4.} {\it $t \ge 5$.}
\medskip
If $n_1\ge 2$, then $\tilde{r}$ is general,
and the rational number $\tilde{\tilde{r}}$ is as in
the case for $t-2$ or $t-1$ according to whether $n_2=1$ or $n_2\ge 2$.
Thus this is impossible by the inductive hypothesis.
Hence we have $n_1=1$
and so let $\tilde{r}':=1-\tilde{r}=[n_2+1,n_3,\cdots, n_t]$.
Note that $\tilde{r}'$ is general
and the rational number $\tilde{\tilde{r}}'$ is as in the case for $t-3$ or $t-2$
according to whether $n_3=1$ or $n_3\ge 2$.
Thus this is impossible by the inductive hypothesis.

Main Theorem~\ref{main_theorem} is now completely proved.
\end{proof}

\section{Proof of Theorems~\ref{main_corollary} and \ref{main_corollary2}}
\label{sec:proof_of_main_corollary}

Consider a hyperbolic $2$-bridge link $K(r)$ with $0<r \le 1/2$,
and assume that the loop $\alpha_s$ with $s\in I_1(r) \cup I_2(r)$
is either peripheral or imprimitive.
Then, by \cite[Lemma~7.2]{lee_sakuma_3},
there is a nontrivial element $w\in G(K(r))$
such that $w\not\in \langle u_s\rangle$
and $wu_sw^{-1}=u_s$.
This identity cannot hold in $F(a,b)$,
since $u_s$ is not a nontrivial cyclic permutation of itself.
So by \cite[Lemma~7.1]{lee_sakuma_3},
the identity $wu_sw^{-1}=u_s$ in $G(K(r))$
is realized by a nontrivial reduced annular $R$-diagram, $M$,
with outer and inner labels $u_s$ and $u_s^{-1}$, respectively.
Then $M$ satisfies the assumption of \cite[Theorem~4.11]{lee_sakuma_2}
and hence its conclusion.

If $r=[2,n]$ or $[2,1,n]$ for some $n\ge 2$, then
Theorems~\ref{main_corollary} and \ref{main_corollary2}
are already proved
in \cite[Section~7]{lee_sakuma_3},
where all possible diagrams $M$ are described.
We note that we can observe that
all such diagram contain a mixing vertex of degree $4$
(see \cite[Figures~13, 14(b), 15(b) and 16]{lee_sakuma_3}).

If $r=[n,2]$ or $[n,1,2]$ with $n\ge 2$, then
the link $K(r)$ is equivalent to the link whose slope is
of the previous type, and so
Theorems~\ref{main_corollary} and \ref{main_corollary2}
in this case are deduced from the results in \cite[Section~7]{lee_sakuma_3}.

If $r=[m,n]$ or $[m,1,n]$ with $m,n\ge 3$,
then by the results in Sections~\ref{sec:proof_of_main_theorem_1}
and~\ref{sec:proof_of_main_theorem_3},
we see that there are no such diagrams.
So, any $\alpha_s$ with $s\in I_1(r) \cup I_2(r)$
is neither peripheral nor imprimitive in this case.

Finally, suppose that $r$ is general.
Then, by the proof of
Propositions~\ref{prop:induction_general_1} and
~\ref{prop:induction_general_3},
we can construct from $M$
a non-trivial reduced conjugacy diagram $\tilde M$
over $G(K(\tilde r))$
with outer and inner labels $u_{\tilde s}$ and $u_{\tilde s}^{-1}$, respectively,
for some $\tilde s\in I_1(\tilde r)\cup I_2(\tilde r)$,
such that $\tilde M$ contains a mixing vertex of degree $4$.
However, we can inductively show that this is impossible
by using the preceding results.
Hence, any $\alpha_s$ with $s\in I_1(r) \cup I_2(r)$
is neither peripheral nor imprimitive in this case.

This completes the proof of
Theorems~\ref{main_corollary} and \ref{main_corollary2}.
\qed

\bibstyle{plain}

\end{document}